\documentclass[10pt]{amsart}

\usepackage{amsmath,amsfonts,amsthm,graphicx, color}

\title[Diffeomorphisms with positive metric entropy]{Diffeomorphisms with positive metric entropy}
\author[A. Avila, S. Crovisier, A. Wilkinson]{A. Avila, S. Crovisier, and A. Wilkinson}
\date{\today}

\thanks{A.A. was partially supported by the ERC Starting Grant Quasiperiodic.
A.A. and S.C. were partially supported by the Balzan Research Project of J.Palis.
A.W. was supported by NSF grant DMS-1316534.
}

\address{Artur Avila \newline
\rm CNRS, IMJ-PRG, UMR 7586, Univ. Paris Diderot, Sorbonne Paris Cit\'e,
Sorbonne Universit\'es, UPMC Univ. Paris 06, F-75013, Paris, France \&
IMPA, Estrada Dona Castorina 110, Rio de Janeiro, Brazil.}
\address{Sylvain Crovisier \newline
\rm CNRS - Laboratoire de Math{\'e}matiques d'Orsay, UMR 8628\newline
Universit{\'e} Paris-Sud 11, 91405 Orsay Cedex, France.}
\address{Amie Wilkinson \newline
\rm Department of Mathematics,\newline
University of Chicago, 5734 S. University Avenue Chicago, Illinois 60637, USA.}
\theoremstyle{plain}
\newtheorem{theorem}{Theorem}[section]
\newtheorem{lemma}[theorem]{Lemma}

\newtheorem{corollary}[theorem]{Corollary}
\newtheorem{proposition}[theorem]{Proposition}

\newtheorem{conjecture}{Conjecture}

\newtheorem{quest}[conjecture]{Question}

\newtheorem*{corollary*}{Corollary}
\newtheorem*{claim}{Claim}

\newtheoremstyle{vThm*}%
{}{}%
{\itshape}%
{-3pt}{\bfseries}%
{}{ }%
{\thmnote{#3}}%
\theoremstyle{vThm*}
\newtheorem*{nThm*}{}

\theoremstyle{definition}

\newtheorem*{definition*}{Definition}

\def\diam{\mathrm{diam}}

\def\Diff{\operatorname{Diff} }
\def\diff{\operatorname{Diff} }

\def\title{\em}

\def\id{\operatorname{Id}}
\def\phcs{{H^s_\mathrm{Pes}}}
\def\phcu{{H^u_\mathrm{Pes}}}
\def\phc{{H_\mathrm{Pes}}}
\def\Leb{m}
\def\jac{\operatorname{Jac}}
\def\Jac{\operatorname{Jac}}

\def\Nuh{\operatorname{Nuh}}
\def\supp{\operatorname{supp}}

\def\SL{\operatorname{SL}}

\def\cK{\mathcal{K}}

\def\cU{\mathcal{U}}

\def\cG{\mathcal{G}}

\def\cV{\mathcal{V}}

\def\Homeo{\operatorname{Homeo}}

\def\transverse{\,\raise2pt\hbox to1em{\hfil$\top$\hfil}\hskip -1em \hbox
to1em{\hfil$\cap$\hfil}\,} 
\def\be{\begin{equation}}
\def\ee{\end{equation}}

\def\cG{\mathcal{G}}
\newcommand\R{\mathbb R}
\newcommand\RR{{\mathbb R}}

\newcommand\Z{{\mathbb Z}}
\newcommand\NN{{\mathbb N}}
\newcommand\N{{\mathbb N}}
 
\newlength{\figboxwidth} 

\makeatletter
\let\original@addcontentsline\addcontentsline
\newcommand{\dummy@addcontentsline}[3]{}
\newcommand{\DeactivateToc}{\let\addcontentsline\dummy@addcontentsline}
\newcommand{\ActivateToc}{\let\addcontentsline\original@addcontentsline}
\makeatother

\begin{document}

\begin{abstract}

We  obtain a dichotomy for $C^1$-generic, volume-preserving diffeomorphisms:
either all the Lyapunov exponents of almost every point vanish
or the volume is ergodic and non-uniformly Anosov
(i.e. nonuniformly hyperbolic and the splitting into stable and unstable
spaces is dominated). This completes a program first put forth by Ricardo Ma\~n\'e.
\end{abstract}

\maketitle
\addcontentsline{toc}{section}{\mbox{}\quad\quad Introduction}
\DeactivateToc
\section*{Introduction}

From a probabilistic perspective, ergodicity is the most basic irreducibility property of a dynamical system.
A measurable map $f\colon M\to M$ is {\em ergodic} with respect to an invariant probability measure $\mu$ if every $f$-invariant subset of $M$ is $\mu$-trivial: $f^{-1}(A) = A$ implies $\mu(A) = 0$ or $1$, for every measurable $A\subset M$.   In the context of this paper, where $M$ is a compact manifold, $f$ is a homeomorphism, and $\mu=m$ is a normalized volume, ergodicity is equivalent to equidistribution of almost every orbit: for $m$-almost every $x\in M$ and every continuous $\phi\colon M\to \RR$,
\[
\lim_{n\to \infty} \frac{1}{n} \sum_{j=1}^{n} \phi(f^j(x)) = \int_{M}\phi\, d\operatorname{m}.
\]

Is ergodicity with respect to volume a typical property? The question
was first addressed by Oxtoby and Ulam
in the 1930's \cite{OU},
who proved that the {\em generic} volume-preserving homeomorphism is ergodic; that is,
the set of ergodic maps  in the space $\Homeo_\mathrm{vol}^+(M)$ of volume-preserving homeomorphisms
contains a countable intersection of open and dense sets in the uniform topology.  A natural question, still open in general, is whether such a result extends to the space of volume-preserving {\em diffeomorphisms}.

If one looks at the other extreme of regularity, $C^\infty$ diffeomorphisms, ergodicity is not a typical property at all: KAM theory guarantees on any manifold of dimension at least $2$ an {\em open} set of diffeomorphisms  in $\Diff^\infty_\mathrm{vol}(M)$ that are not ergodic.  This paper focuses on the lowest class of differentiability, $C^1$ diffeomorphisms, where the question is still open: {\em is ergodicity a generic property in the space $\Diff^1_\mathrm{vol}(M)$ of $C^1$ volume-preserving diffeomorphisms of a compact manifold $M$?}  

As a first approach  to this question, one should ask 
whether the techniques of the Oxtoby-Ulam proof can be extended 
to the $C^1$ setting.  There is an immediate obstruction: metric entropy.  The same technique (namely periodic approximation) that proves genericity of ergodicity in \cite{OU} also proves that the metric entropy $h_{m}(f)$  of a generic
 $f\in \Homeo_\mathrm{vol}^+(M)$ is $0$.  The corresponding statement is false for 
$\Diff^1_\mathrm{vol}(M)$, as we explain below: there are {\em open} sets of diffeomorphisms  $f\in\Diff^1_\mathrm{vol}(M)$ with $h_{m}(f)>0$.
{\color{black} Thus the Oxtoby-Ulam technique cannot be na\"ively extended from the $C^0$-category to prove general results about  $C^1$-generic diffeomorphisms.}

This phenomenon of robustly positive entropy is
most clearly demonstrated by the  Anosov maps, in which every direction
in the tangent bundle to $M$ sees expansion or contraction under iteration of the
derivative $Df^n$.
Interestingly, 
this {\em uniformly hyperbolic} behavior that gives rise to positive metric
entropy in Anosov systems is also the source of a powerful mechanism for 
ergodicity, known as the Hopf argument~\cite{Anosov}, which is of a very different nature than the Oxtoby-Ulam mechanism.
Here we show for generic diffeomorphisms in
 $\Diff^1_\mathrm{vol}(M)$,   positive metric entropy is associated with 
a strong type of non-uniformly hyperbolic behavior, which we call \emph{non uniformly Anosov}. 
Harnessing this nonuniform hyperbolicity, we prove:

\begin{nThm*}{{\bf Theorem A.}}
$C^1$-generically, a volume-preserving diffeomorphism $f\colon M\to M$ of a compact
manifold $M$ with positive entropy is ergodic.
\end{nThm*}

Our proof of this theorem completes a program first put forth by Ricardo Ma\~n\'e to understand  the Lyapunov exponents of volume-preserving diffeomorphisms from a $C^1$-generic perspective.  In his 1983 ICM address~\cite{M},
Ma\~n\'e announced the following remarkable result, whose proof was later completed by Bochi~\cite{B}.

\begin{nThm*}{{\bf Theorem. (Ma\~n\'e-Bochi)}}
$C^1$-generically, an  area preserving diffeomorphism $f$ of a compact
connected surface $M$ is either Anosov (and ergodic) or satisfies
\[\lim_{n\to\pm \infty} \frac{1}{n} \log\| D_x f^nv\|
= 0,\]
for a.e. $x\in M$ and every  $0 \neq v \in T_xM$.
\end{nThm*}

Our main result gives the optimal generalization to higher dimensions:

\begin{nThm*}{{\bf Theorem B.}}
$C^1$-generically, a volume-preserving diffeomorphism $f$ of a compact connected
manifold $M$ is either nonuniformly Anosov and ergodic or satisfies
\[\lim_{n\to\pm \infty} \frac{1}{n} \log\| D_x f^nv\|
= 0\]
for a.e. $x\in M$ and every  $0 \neq v \in T_xM$.
\end{nThm*}

Theorem B  was conjectured in its present form by Avila-Bochi \cite {AB} where it was shown that
generic diffeomorphisms in $\Diff^1_\mathrm{vol}(M)$ with only non-zero Lyapunov
exponents almost everywhere are ergodic and non-uniformly Anosov.
In dimension three, Theorem B was proved by M.A. Rodr\'\i guez-Hertz \cite
{jana} by reducing to an analysis of dominated splittings admitting some
uniformly hyperbolic subbundles, which have been thoroughly described
for $3$-manifolds.  Our proof of Theorem B in the general case follows a very
different route, focused on the elimination of zero Lyapunov exponents
throughout large parts of the phase space.

{\color{black} In another paper~\cite{ACW}, we will use Theorem B above in order to prove a $C^1$-version of a conjecture
by Pugh and Shub: among smooth partially hyperbolic volume-preserving diffeomorphisms,
the stably ergodic ones are $C^1$-dense.}
\bigskip

Before exploring {\color{black} further consequences} of Theorems A and B, we put it in context and explain the terminology.  Throughout, $M$ will denote a closed connected Riemannian manifold with dimension $d$,  and $\Diff^r(M)$ will denote the set of $C^r$ diffeomorphisms of $M$ endowed with the $C^r$-topology.
The volume induces, after normalization, a Borel probability measure $m$ and
we denote by  $\Diff^r_\mathrm{vol}(M)$  the set of $f\in \Diff^r(M)$ preserving $m$.
Both  $\Diff^r(M)$ and $\Diff^r_\mathrm{vol}(M)$
are Baire spaces.  We say that 
a property of (volume-preserving) diffeomorphisms is {\em $C^r$ generic} if it holds
on a dense $G_\delta$ (i.e., a countable intersection of open-dense sets) in  $\Diff^r(M)$ (respectively $\Diff^r_\mathrm{vol}(M)$).

A  measure of chaoticity for volume-preserving diffeomorphisms is given by the notion of
Lyapunov exponents.
A real number $\chi$ is a \emph{Lyapunov exponent} of $f$ at $x\in M$ if there exists a nonzero vector
$v\in T_xM$ such that
\begin{equation}\label{e=lyaplim}
\lim_{n\to\pm\infty} \frac{1}{n} \log \|Df^n(v)\| = \chi.
\end{equation}
Oseledets's ergodic theorem implies that there is a set $\Omega\subset M$ of total measure
 -- i.e., $\mu(\Omega) = 1$, for every invariant  Borel probability measure $\mu$ --
 {\color{black} with the following property: for any $x\in \Omega$ there exists $\ell(x)\geq 1$ and 
and a $Df$-invariant splitting
\begin{equation}\label{e=oseledec}
T_x M = E_1(x)\oplus E_2(x)\oplus \cdots \oplus E_{\ell(x)}(x), 
\end{equation}
depending measurably on $x$}
such that  the limit 
$\chi=\chi(x,v)$ in \eqref{e=lyaplim}
exists for every every $v\in E_i(x)\setminus \{0\}$.
The value $\chi(x,v)$ is constant in $E_i(x)\setminus \{0\}$
so that $\chi(x,\cdot)$ can assume at most $\dim(M)$ distinct values $\chi_1(x), \ldots, \chi_{\ell(x)} (x)$.
If $f$ preserves the volume $m$, then the sum of the Lyapunov exponents is zero
{\color{black} on a set of total }  {\color{black} measure}.

Lyapunov exponents can be used to control a more familiar barometer of chaos, namely the metric (or measure-theoretic) entropy.   Entropy and Lyapunov exponents of $C^1$ diffeomorphisms are related by Ruelle's inequality, which states that for $f\in\Diff^1(M)$ preserving a Borel probability $\mu$,
\begin{equation*}
h_\mu(f) \leq  \int_M  \sum_{\chi_i(x) \geq 0} \dim(E_i(x)) \chi_i(x) \, d\mu(x).
\end{equation*}
For $\mu=m$, the reverse equality was proved by Pesin for all $f\in \Diff^2_\mathrm{vol}(M)$  and  generically
in $\Diff^1_\mathrm{vol}(M)$  by Tahzibi~\cite{tahzibi1} and Sun-Tian~\cite{sun-tian}.
In particular for generic $f\in \Diff^1_\mathrm{vol}(M)$, the metric entropy vanishes exactly
when the second case of Theorem B occurs. Hence Theorem B implies Theorem A.
\bigskip

In his 1983 address mentioned above, Ma\~n\'e proposed to study how the 
 ``Oseledets splitting"
(\ref{e=oseledec}) varies as a function of the diffeomorphism $f$,
in the $C^1$ topology.
A diffeomorphism is Anosov 
if there exists
a continuous $Df$-invariant splitting
\begin{equation}\label{e=Anosov}
TS = E^u\oplus E^s
\end{equation}
and $0<\lambda < 1$,  $n_0 \in \N$,
such that $\|Df^n|E^u\| \leq \lambda^n$ and $\|(Df^n|E^s)^{-1}\| \leq
\lambda^n$ for every $n \geq n_0$. 
In this case, the (measurable) Oseledets splitting (\ref{e=oseledec}) refines the (continuous) Anosov 
splitting (\ref{e=Anosov}) and the Lyapunov exponents are nonzero (either smaller than $-|\log(\lambda)|$
or larger than $+|\log(\lambda)|$).
This property is extremely rigid in low dimension (and conjecturally rigid in all dimensions): in particular, if $f$ is an Anosov diffeomorphism of a surface,
then $M$ is a torus, and $f$ is topologically conjugate to a hyperbolic linear automorphism.  Thus
the Ma\~n\'e-Bochi theorem implies that if $M$ is not a torus, then the $C^1$-generic area-preserving diffeomorphism of $M$ has metric entropy~0.

However, in higher dimensions {\em uniform} hyperbolicity is too much to aim for:
 any volume-preserving diffeomorphism admitting a
{\it dominated splitting} must
have robustly positive metric entropy.
 A diffeomorphism $f \in \Diff^1(M)$ is said
to admit a (global) dominated splitting if there exists a continuous
non-trivial decomposition $TM=E_1\oplus E_2$ that is $Df$-invariant and
satisfies {\color{black}  \[ \|(Df^N|E_1)^{-1}\| \|Df^N|E_2\| <1,\]}
for some $N \in \N$.  Thus $f$ is an Anosov map if and only if it admits a
uniformly hyperbolic dominated splitting. 

While in dimension $2$ a dominated
splitting for an area-preserving diffeomorphism is always Anosov,
already in dimension $3$ there are
manifolds that do not support Anosov dynamics, but which are compatible
with a dominated splitting\footnote{\color{black} On the unit tangent bundle of a hyperbolic surface,
the geodesic flow is Anosov; hence
its time-one map is a diffeomorphism preserving a dominated splitting.
However this manifold does not support any Anosov diffeomorphism,
since, in dimension $3$, only the torus has this property
(Franks-Newhouse theorem~\cite{franks,newhouse}).}.

In the presence of robust obstructions to
uniform hyperbolicity, the best one can hope for
is to obtain a dominated splitting $TM=E^+\oplus E^-$ that is {\em non-uniformly hyperbolic}, 
in the sense that  there exists $\chi_0>0$ such that  for $m$-a.e. $x\in M$,
each Lyapunov exponent is either smaller than $\chi_0$ or larger than $\chi_0$.
This leads to:

\begin{definition*} A diffeomorphism $f\in \Diff^1_\mathrm{vol}(M)$ admitting a non-uniformly hyperbolic dominated splitting will be called {\it
non-uniformly Anosov}. 
Equivalently, $f$ is non-uniformly Anosov if it possesses a dominated splitting $TM=E^+\oplus E^-$
and if there exists $0<\lambda<1$ such that for $\Leb$-almost every $x \in M$,
there exists $n_0(x) \in \N$ such that
$\|Df^n(x)|E^-(x)\| \leq \lambda^n$ an $\|Df^{-n}(x)|E^+(x)\| \leq \lambda^n$ for every $n \geq n_0(x)$. 
\end{definition*}

{\color{black} The class of non-uniformly Anosov diffeomorphisms is strictly larger than the Anosov class;}
{\color{black}Shub and Wilkinson~\cite{SW} constructed an open set of non-uniformly Anosov diffeomorphisms
in $\Diff^2_\mathrm{vol}(\mathbb{T}^3)$ that are not Anosov. (Their construction is at the root of one of the arguments
used in this paper; see Section~\ref{s.technique}.)}

The existence of a dominated splitting is a robust dynamical
property (i.e., stable under perturbations in $\Diff^1(M)$), as is uniform
hyperbolicity.  A striking consequence of Theorem~B {\color{black} (proved in Section~\ref{ss.robustness})}    is thus:

\begin{nThm*}{\bf{Corollary 1.}}
A map $f \in \Diff^1_\mathrm{vol}(M)$ has robust positive metric entropy  if and only if
it admits a dominated splitting.
\end{nThm*}

These results highlight the unique features of the  $C^1$ topology.
At least conjecturally, sufficiently regular
volume-preserving diffeomorphisms are expected
to be compatible with a quite different phenomenon: the
coexistence of quasiperiodic behavior (where  Lyapunov exponents vanish)
with chaotic, non-uniformly hyperbolic behavior (inducing positive metric entropy).  Even on surfaces, this
{\color{black} problem} remains open.

\ActivateToc
\addcontentsline{toc}{subsection}{\mbox{}\quad\quad\quad Discussion and questions}
\DeactivateToc
\subsection*{Discussion and questions}

We return briefly to the question posed at the beginning of the paper: {\em 
Is ergodicity a generic property in $\Diff^1_\mathrm{vol}(M)$?}
Some partial results are known.  Bonatti and Crovisier proved~\cite{BC} that transitivity (i.e., existence of a dense orbit) is a generic property in $\Diff^1_\mathrm{vol}(M)$
(the topological mixing also holds~\cite{AC}).  A property in between transitivity and ergodicity with respect to volume is {\em metric transitivity}, where almost every orbit is dense.  A weaker question is thus:

\begin{quest} Is metric transitivity generic in $\Diff^1_\mathrm{vol}(M)$?
\end{quest}

The next question relates to the Oxtoby-Ulam technique \cite{OU}.  If $f$
has entropy $0$, then the results in \cite{BDP},  \cite{BC}
and \cite{A} show that it can be perturbed to have
a dense set of periodic balls.  
\begin{quest} Can every $f\in \Diff_m^1(M)$ of entropy $0$ be $C^1$ approximated by an almost everywhere periodic diffeomorphism (i.e. a diffeomorphism whose periodic points
have full measure)?
\end{quest}

In the case of $C^1$-generic diffeomorphisms with positive entropy, a next goal would be to describe better their 
measurable dynamics.  Some additional argument gives the following corollary of Theorem B,
which is proved in Section~\ref{ss.weak-mixing}:

\begin{nThm*}{{\bf Corollary 2.}}
The generic $f \in \Diff^1_\mathrm{vol}(M)$ with positive metric entropy is weakly mixing.
\end{nThm*}

Due to the lack of regularity in  $\Diff^1_\mathrm{vol}(M)$ we cannot use Pesin theory  to  get the Bernoulli
property to be generic.

\begin{quest}
Are generic nonuniformly Anosov diffeo\-morphisms in $\Diff_m^1(M)$ \\ Bernoulli?
or at least strongly mixing?
\end{quest}

Even among the class of $C^1$ Anosov diffeomorphisms, genericity of the mixing
condition is an open question.
\medskip

{\color{black} When $M$ is endowed with a symplectic form $\omega$, one may also consider the space of diffeomorphisms
$\Diff^1_\omega(M)$ that preserve $\omega$. For ``technical reasons,"  Ma\~n\'e focuses on this case in~\cite{M}:
the symplectic rigidity imposes some symmetry in the Oseledets splitting. The argument developed in the present paper
(Theorem C below) can not be transposed in this setting. Some partial results have been obtained for
$C^1$-generic symplectomorphisms: for instance~\cite{ABW} proves that if there exists an invariant global dominated splitting,
then the volume is ergodic (but it is not non-uniformly hyperbolic, unless the diffeomorphism is Anosov).
}{\color{black} In an upcoming work we will prove the symplectic version of Theorem A, using different (and simpler!) methods that are special to the symplectic setting.}

\ActivateToc

%%%%%%%%%%%%%%%%%%%%%%%%%%%%%%%%%%%%%%%%%%%%%%%%%%%%%%%%%%
%%%%%%%%%%%%%%%%%%%%%%%%%%%%%%%%%%%%%%%%%%%%%%%%%%%%%%%%%%

\section{The main technique: localized, pointwise perturbations of central Lyapunov exponents}\label{s.technique}
From the development of Pesin Theory and the gradual taming of
(sufficiently regular)
non-uniformly hyperbolic dynamics which followed (\cite {Pe}, \cite {katok}),
it has been a central problem to understand how often such systems arise.
While it is understood that the ``opposite'' behavior, the vanishing of all
Lyapunov exponents, does appear robustly (through the KAM mechanism),
it has been proposed by Shub and Wilkinson (\cite {SW}, Question 1a)
that for typical orbits of a generic $C^r$ conservative
dynamical system, the presence of some non-zero Lyapunov exponent implies
in fact that all Lyapunov exponents are non-zero.  Such an optimistic
picture was motivated by an argument, introduced in the same paper, which
allows one to leverage (in a particularly controlled setting)
the non-zero Lyapunov exponents to ``perturb away''
the zero Lyapunov exponents.

The specific situation considered by Shub and Wilkinson consisted of a
trivial circle extension of a linear Anosov map.  This is a partially
hyperbolic dynamical system with a one-dimensional central direction along
which the Lyapunov exponent vanishes everywhere.
Through a carefully designed perturbation, the central
bundle borrows some of the hyperbolicity from the uniformly
expanding bundle, so the {\em average} central Lyapunov exponent becomes
positive.  In order to show that the actual Lyapunov exponent along the
center is non-zero almost everywhere, they observe that the system can be,
at the same time, made ergodic by a separate argument (based on the
Pugh-Shub ergodicity mechanism).

This argument has been pursued further, in low regularity,
by Baraviera and Bonatti \cite {BB}. 
They consider conservative
diffeomorphisms admitting a dominated splitting $TM=E_1 \oplus
\cdots \oplus E_k$ and show that the average of the
sum of the Lyapunov exponents along any subbundle can be made non-zero by
a $C^1$ perturbation.  This result was used by Bochi, Fayad and Pujals
in \cite {BFP} to show that stably ergodic diffeomorphisms, which admit a
dominated splitting by \cite {BDP}, can be
made non-uniformly hyperbolic by perturbation.
\bigskip

In a sense, here we do just the opposite of
\cite {BFP}: we show the generic absence of zero
Lyapunov exponents almost everywhere (under the positive entropy assumption)
is a means to conclude ergodicity (via \cite {AB}).
In order to do this, we must
develop a perturbation argument that can affect directly the actual
Lyapunov exponents of
certain orbits inside an invariant region, and not just their averages over the whole manifold.
Without an assumption of ergodicity, these can be different.
This is obtained through the following local, pointwise version of Bonatti-Baraviera's argument~\cite{BB} (see also \cite{SW}). Even for a diffeomorphism that preserves a globally partially hyperbolic structure,
this is a new result. A more precise statement will be given in Section~\ref{s.local}.

 If $\mu$ and $\nu$ are finite Borel measures on $M$,
the notation $\mu\leq \nu$ means that $\mu(A)\leq \nu(A)$ for all measurable sets $A$.
 For $f \in \Diff^1_m(M)$, $x \in M$, and a nontrivial subspace $F \subset T_x M$,
we denote by $\Jac_F(f,x)$ the Jacobian of $Df$ restricted to $F$, i.e., the product of the singular values of $Df(x)|F$.\\

\begin{nThm*}{{\bf Theorem C.}}
Let $f\in \diff^1_m(M)$, and let $K\subset M$ be an invariant compact set such that:
\begin{itemize}
\item $K$
admits a dominated splitting $T_{K}M=E_1\oplus E_2\oplus E_3$ into three non-trivial subbundles;
\item for almost every point $x\in K$ one has
$$\limsup_{n\to \pm \infty} \frac {1} {n} \log \Jac_{E_2(x)}(f^n,x)\leq 0.$$ 
\end{itemize}
Then for every $\varepsilon>0$
and every small neighborhood $Q$ of $K$,
there exists a diffeomorphism $g$ arbitrarily close to $f$ in $\diff^1_m(M)$
such that for every $g$-invariant measure $\nu$ such that $\nu\leq m|Q$
and $\nu(M)\geq \varepsilon$,
one has $\int \log \Jac_{E_2(g,x)}(g,x)d\nu(x)<0$.
\end{nThm*}

In the previous statement
{\color{black} the fibers of the bundles $E_1,E_2, E_3$ do not necessarily have constant
dimension, but one can easily reduce the theorem to this case by decomposing the compact set $K$.}
The expression
$E_1(g)\oplus E_2(g)\oplus E_3(g)$
denotes the \emph{continuation of the dominated splitting} for the diffeomorphism $g$
on any $g$-invariant set contained in a neighborhood of $Q$. (See Section~\ref{ss.continuation}.)
\bigskip

We remark that the
existence of a global dominated splitting, which is a starting point
in \cite {SW} and \cite {BB}, is here also obtained as a consequence of
non-uniform hyperbolicity (again, via \cite {AB}).  The hypothesis of
positive entropy (and hence the existence of some non-zero Lyapunov
exponents) is however enough to obtain local dominated splittings, thanks to a
result of Bochi and Viana \cite {BV} who showed
that, for almost every orbit of generic conservative diffeomorphisms,
the Oseledets splitting extends continuously
to a dominated splitting on its closure.

Our basic technique is the following.  First,  we may assume that
the initial (generic) diffeomorphism has a  positive measure set $K$ of orbits
having some, but not all, non-zero Lyapunov exponents
(otherwise \cite {AB} yields the conclusion at once).
Consider a sufficiently long segment
of a typical orbit
that admits a dominated splitting $E_1 \oplus E_2 \oplus
E_3$, where $E_2$ corresponds to zero Lyapunov exponents.
If this orbit segment is long enough, then it ``sees'' the Lyapunov
exponents of the orbit.  We can then reproduce the perturbation technique
of \cite {SW} and \cite {BB} along the orbit: since this technique concerns average
exponents, we first thicken the initial point to a small
positive measure set, and conclude that the average of the sum of the
Lyapunov exponents along  the central bundle can be decreased.  In order to produce a pointwise estimate, we use a randomization technique introduced by
Bochi in \cite {Bsymp}, which allows us to apply the Law of Large Numbers to
promote  the averaged estimate to a pointwise one.  Using a standard  towers
technique, this argument can be
carried out simultaneously a large set of the orbits remaining within 
the domain of definition $U$ of the local dominated splitting.

Naturally, the perturbation changes the dynamics, so in principle the
decrease of the sum of Lyapunov exponents could be cancelled later.
In fact
the dynamics could change so much that many orbits escape $U$ and we lose
all control, but this ``loss of mass'' is an irreversible event and
thus relatively harmless.  As for possible cancellations, we simply assume
away the problem by restricting attention to
the case where the Lyapunov exponents
along $E_2$ are non-positive for almost every orbit that remains within $U$. 
Remarkably, this seemingly very strong hypothesis can be in fact verified
along the steps of a carefully designed inductive argument.  In any case,
with this assumption we can conclude directly that for most orbits remaining
in $U$ the number of zero Lyapunov exponents is strictly less than the
dimension of $E_2$, after perturbation.

Iterating this argument, we eventually succeed in either eliminating all
non-zero Lyapunov exponents, or in obtaining vanishing Lyapunov exponents
almost everywhere (this happens when we keep running into the situation
where orbits escape the domains of definition of local dominated
splittings).

%%%%%%%%%%%%%%%%%%%%%%%%%%%%%%%%%%%%%%%%%%%%%%%%%%%%%%%%%%
%%%%%%%%%%%%%%%%%%%%%%%%%%%%%%%%%%%%%%%%%%%%%%%%%%%%%%%%%%

\section{A dichotomy for conservative diffeomorphisms}
In this section we prove Theorem  B  assuming Theorem C.

\subsection{Dominated splittings and center Jacobians}
Let $f \in \Diff^1(M)$. We recall well-known properties of dominated splittings.
As before $d=\dim(M)$.

\subsubsection{} Given an $f$-invariant compact set $K_f$,
we say that $f|K_f$ admits a \emph{dominated splitting of type
$(d_1,d_2,d_3)$} (where $d_1,d_2,d_3 \geq 0$ and $d_1+d_2+d_3=d$)
if there is an $f$-invariant splitting
$T_x M=E_1(x) \oplus E_2(x) \oplus E_3(x)$ defined over $K$,
where $E_*(x)= E_*(f,x)$ are subspaces of dimension $d_*$, $*=1,2,3$, and
there is an $n \in \N$ such that for each $x\in K$ one has
$$\|(Df^n(x)|E_1(x))^{-1}\|<\|Df^n(x)|E_2(x)\|^{-1},$$
$$\|(D f^n(x)|E_2(x))^{-1}\|<\|Df^n(x)|E_3(x)\|^{-1}.$$
In other words, the smallest contraction along $E_1$ and the largest expansion along $E_3$
dominate the behavior along $E_2$.
{\color{black} For a fixed $d_1, d_2, d_3$}, the $E_*(x)$ are uniquely defined in this
way and depend continuously on $x$.

\subsubsection{}\label{ss.continuation} This dominated splitting is robust in the following sense.  Consider an
arbitrary continuous extension of the $E_*(x)$ to a neighborhood of $K_f$
and consider arbitrary metrics on the Grassmanian manifolds of $M$.
Then for every $\alpha>0$, there are
neighborhoods $\cV \subset \Diff^1(M)$ of $f$ and
$V \subset M$ of $K_f$ such that if $g \in \cV$ and $K_g \subset V$ is a
compact invariant set, then $g|K_g$ admits a dominated splitting of type
$(d_1,d_2,d_3)$, and moreover the spaces $E_*(g,x)$ are $\alpha$-close
to (the extension of) $E_*(f,x)$ for every $x \in K_g$.

\subsubsection{}
Given a compact set $Q \subset M$, we let
$K(f,Q)=\bigcap_{n \in \Z} f^n(Q)$
be its maximal $f$-invariant subset.  Notice that
$K(g,Q) \subset V$ for every neighborhood $V$ of $K(f,Q)$ and every
$g \in \Diff^1(M)$ close to $f$ in the $C^0$ topology.

The previous paragraph thus implies that
\emph{the set of all $g\in \Diff^1(M)$ such that
$g|K(g,Q)$ admits a dominated splitting of type $(d_1,d_2,d_3)$ is open.}
{\color{black} This includes the diffeomorphisms $g$ such that $K(g,Q)$ is empty.}

\subsubsection{}\label{ss.oseledets}
Let $X_f\subset M$ be the set of \emph{Oseledets regular points} $x$ of $f$, i.e.
which have well-defined Oseledets splitting and Lyapunov exponents
$$\lambda_1(f,x)\geq \lambda_2(f,x)\geq\dots\geq \lambda_d(f,x).$$
By Oseledets's theorem, $X_f$ is a measurable $f$-invariant set
{\color{black} of total measure.}
Moreover, the Lyapunov exponents define $d$ functions $\lambda_1,\dots,\lambda_d\in L^1(\mu)$.

For any regular point $x$, by summing all the directions associated to
the positive, zero, or negative Lyapunov exponents, we obtain a splitting:
$$T_xM=E^+(x)\oplus E^0(x)\oplus E^-(x).$$
The dimensions $\dim(E^+(x))$, $\dim(E^-(x))$ are called \emph{unstable} and \emph{stable dimensions}
of $x$.

An invariant probability measure is \emph{hyperbolic} if for almost every point
the Lyapunov exponents are all different from zero.

\subsubsection{}
For $x \in M$ and a subspace $F \subset T_x M$, we let
$$\Delta_F(f,x)={\color{black} \lim_{n\to \pm \infty}} \frac {1} {n} \log \Jac_F(f^n,x),$$
{\color{black} which is well-defined on a set of $x$ of total measure.}
If $x$ is Oseledets regular, and $F$ is a sum of
Oseledets subspaces, then $\frac {1} {n} \log \Jac_F(f^n,x)$  converges
to the sum of the Lyapunov exponents of $f$ along $F$. 
Moreover if $\nu$ is an $f$-invariant finite Borel measure, and $F(x)\subset T_xM$
is a measurable $f$-invariant distribution of subspaces defined $\nu$-almost everywhere,
then for every $n \geq 1$ we have
\begin{equation*}
\int \Delta_{F(x)}(f,x) d\nu(x)=\frac {1} {n} \int \log \Jac_{F(x)}(f^n,x)
d\nu(x).
\end{equation*}

\subsubsection{}
Recall that if $\mu$ and $\nu$ are finite Borel measures, the notation $\mu\leq \nu$ means that $\mu(A)\leq \nu(A)$ for all measurable sets $A$.  This property is equivalent to the two conditions: 
$\mu$ is absolutely continuous with respect to $\nu$ and the Radon-Nikodym derivative $d\mu/d\nu$ is essentially bounded above by $1$.  When $\nu$ is fixed, the set of measures $\mu$ satisfying $\mu\leq \nu$ is clearly compact in the weak-$\ast$ topology.

\subsubsection{} Recall that $m$ is a smooth volume on $M$.
For $\varepsilon>0$ and $Q\subset M $ compact, we denote by $\cG_\varepsilon(Q,d_1,d_2,d_3)$ the set of all $g \in\diff^1(M)$
such that
\begin{itemize}
\item $g|K(g,Q)$ admits a dominated splitting of type $(d_1,d_2,d_3)$ {\color{black} (including the case where $K(g,Q)=\emptyset$)},
\item for every $g$-invariant measure $\nu \leq m|Q$ satisfying
$\nu(M) \geq \varepsilon$, one has
$$
\int \Jac_{E_2(g,x)}(g,x) d\nu(x)<0.
$$
\end{itemize}
{\color{black} The compactness of the set of $\nu$ satisfying $\nu\leq \Leb|Q$ and the openness of the dominated splitting condition give: }
\begin{lemma}\label{l.G-open}

For every $\varepsilon>0$, the set $\cG_\varepsilon(Q,d_1,d_2,d_3)$ is open in $\diff^1(M)$.

\end{lemma}

\begin{proof}
Consider $(g_n)$ converging to $g$ in $\diff^1(M)$ and assume
$g_n \notin \cG_\varepsilon(Q,d_1,d_2,d_3)$. 
We have to prove that $g \notin \cG_\varepsilon(Q,d_1,d_2,d_3)$.
For the sake of contradiction, by Section~\ref{ss.continuation} it suffices to
 assume that the {\color{black} $g_n|K(g_n,Q)$} admit dominated splittings of type $(d_1,d_2,d_3)$.
Let
$\nu_n \leq m|Q$ be a sequence of $g_n$-invariant measures
satisfying $\nu_n(M) \geq \varepsilon$ and $\int \Jac_{E_2(g_n,x)}(g_n,x)
d\nu_n(x) \geq 0$.  Let $\nu$ be a weak-$*$ limit of $\nu_n$.  Then
$\nu \leq \Leb|Q$ is $g$-invariant and satisfies $\nu(M) \geq \varepsilon$ and
$\int \Jac_{E_2(g,x)} d\nu(x) \geq 0$. Hence $g\not\in \cG_\varepsilon(Q,d_1,d_2,d_3)$.
\end{proof}

\subsection{Oseledets blocks}
For $f\in \Diff_m^1(M)$, the set of regular points $X_f$ splits
into $f$-invariant measurable subsets $X_f(d_1,d_2,d_3)$, $d_1+d_2+d_3=d$ and $d_*\geq 0$,
defined as the set of points admitting $d_1$ positive, $d_2$ zero and
and $d_3$ negative Lyapunov exponents (counted with multiplicity).
Note that:
\begin{itemize}
\item $X_f(0,d,0)$ is the set of points whose Lyapunov exponents are all zero;
\item the set of \emph{non-uniformly hyperbolic points}, denoted by $\Nuh_f$
is the union of the sets $X(d_1,0,d_3)$, with $d_1,d_3>0$;
\item by volume preservation, the other non-empty sets satisfy $d_1,d_2,d_3>0$.
\end{itemize}

\subsubsection{Domination.}
Oseledets and dominated splittings coincide generically.

\begin{theorem}[Bochi-Viana~\cite{BV}]\label{t.BV}
For any diffeomorphism $f$ in a dense G$_\delta$ subset of $\diff^1_m(M)$
and for any $\varepsilon>0$, for each Oseledets block $X_f(d_1,d_2,d_3)$
there exists an $f$-invariant compact set $K$ satisfying:
\begin{itemize}
\item $f|K$ admits a dominated splitting of type $(d_1,d_2,d_3)$,
\item $m(X_f(d_1,d_2,d_3)\setminus K)\leq \varepsilon$. 
\end{itemize}
\end{theorem}
{\color{black} In the previous theorem, the set $K$ is not necessarily contained in
$X_f(d_1,d_2,d_3)$.}

\subsubsection{The non-uniformly hyperbolic set.}
Generically the non-uniformly hyperbolic set $\Nuh_f$
coincides $m$-almost everywhere with a single Oseledets block. 

\begin{theorem}[Avila-Bochi~\cite{AB}, Theorem A]\label{t.AB}
For any diffeomorphism $f$ in a dense G$_\delta$ subset of $\diff^1_m(M)$,
either $m(\Nuh_f)=0$ or $\Nuh_f$ is dense in $M$ and the restriction
$m|\Nuh_f$ is ergodic.
\end{theorem}

\subsubsection{The set where all exponents vanish.} As a consequence we get
(see also \cite {AB}, Corollary 1.1):
\begin{corollary}\label{c.vanish}
For any diffeomorphism $f$ in a dense G$_\delta$ subset of $\diff^1_m(M)$,
if $m(\Nuh_f)>0$, then there exists a global dominated splitting
$TM=E\oplus F$ on $M$ such that for $m$-almost every point $x\in \Nuh_f$,
$$
v\in E(x)\setminus\{0\} \implies \lim_{n\to \infty}\frac1n\log\|D_xf^n(v)\| >0,
$$
and
$$
v\in F(x)\setminus\{0\} \implies \lim_{n\to \infty}\frac1n\log\|D_xf^n(v)\| <0.
$$
In particular, $X_f(0,d,0)=\emptyset$.
\end{corollary}
\begin{proof}
{\color{black} Theorem~\ref{t.AB} implies that, $C^1$-generically, 
if $\Nuh_f$ has positive volume, then it is dense in $M$, the restriction of $m$ is ergodic
and it coincides with a set $X(d_1,0,d_3)$. Suppose then
that $\Leb(\Nuh_f)>0$, and let $\varepsilon = \Leb(\Nuh_f)/2$. 
By Theorem~\ref{t.BV},  there exists an invariant compact set $K$ with 
$\Leb(\Nuh_f\setminus K)<\varepsilon$
that admits a non-trivial dominated splitting.  In particular,  $\Leb(\Nuh_f\cap K) > 0$; since $\Leb|\Nuh_f$ is ergodic, this implies that $\Leb(\Nuh_f\setminus K) = 0$. This proves that the compact set $K$ contains $m$-almost every point of $\Nuh_f$, and hence coincides with $M$, since $\Nuh_f$ is dense in $M$. We have thus proved that $M$ has a non-trivial dominated splitting, and
so the set $X_f(0,d,0)$ is empty.}
\end{proof}

\subsubsection{The other Oseledets blocks.}
Using Theorem C we get:
\begin{corollary}\label{c.others}
For any diffeomorphism $f$ in a dense G$_\delta$ subset of $\diff^1_m(M)$,
the Oseledets blocks $X_f(d_1,d_2,d_3)$ with $d_1,d_2,d_3>0$ have volume zero.
\end{corollary}
\begin{proof}
Let $\cK$ be a countable family of compact sets of $M$
such that for any $K\subset U\subset M$, with $K$ compact and $U$
open, there exists $Q\in \cK$ satisfying
$K\subset Q\subset U$.
By Lemma~\ref{l.G-open}, one can assume that for any $Q\in \cK$, any $\varepsilon>0$
such that $1/\varepsilon\in \NN$, and any type $(d_1,d_2,d_3)$,
the diffeomorphism $f$ either belongs to
$\cG_\varepsilon(Q,d_1,d_2,d_3)$ or to
$\Diff^1_\mathrm{vol}(M)\setminus \overline{\cG_\varepsilon(Q,d_1,d_2,d_3)}$.
 \medskip

{\color{black}
\noindent
\paragraph{\bf Case 1. The case $\Nuh_f$ has zero volume.}
We prove by increasing induction on $d_2+d_3$
that $X_f(d_1,d_2,d_3)$ has volume zero, for each triple $(d_1,d_2,d_3)$ with $d_1+d_2+d_3=d$
and $d_1,d_2,d_3>0$.
We thus fix $(d_1,d_2,d_3)$ and assume that $m(X_f(d'_1,d'_2,d'_3))=0$
for each triple $(d'_1,d'_2,d'_3)$ such that $d'_2+d'_3<d_2+d_3$ and $d'_1,d'_2,d'_3>0$.

\begin{claim}
For any set $X_f(d'_1,d'_2,d'_3)$ with positive volume, one has $d'_2+d'_3\geq d_2+d_3$.
\end{claim}
\begin{proof}
We consider separately the three possible cases:
\begin{itemize}
\item $(d'_1,d'_2,d'_3)=(0,d,0)$: the claim holds trivially,
\item $d'_1,d'_2,d'_3$ are all non zero: our inductive assumption implies the claim,
\item $d'_2= 0$: this does not occur since $\Nuh_f$ has zero volume.
\end{itemize}
\end{proof}

We fix $\varepsilon>0$ with $1/\varepsilon\in \NN$.
By Theorem~\ref{t.BV} there exists an invariant compact set $K$ {\color{black} (possibly empty)} 
such that $m(X_f(d_1,d_2,d_3)\setminus K)$ is smaller than $\varepsilon$
and such that $f|K$ admits a dominated splitting $E_1\oplus E_2\oplus E_3$ of type $(d_1,d_2,d_3)$.

Almost every point $x\in K$ belongs to a set
$X_f(d'_1,d'_2,d'_3)$ with positive volume.
By the claim above, $d'_2+d'_3\geq d_2+d_3$.
As a consequence $E_2(f,x)$ is contained
in the sum of the central and the stable spaces of the Oseledets decomposition at $x$.
This implies $\Delta_{E_2(f,x)}(f,x)\leq 0$.

We have proved that the assumptions of Theorem C are satisfied.
We choose a small neighborhood $Q\in \cK$ of $K$.
There exists $g$ arbitrarily close to $f$ in $\Diff^1_\mathrm{vol}(M)$
such that for every invariant measure $\nu\leq m|Q$
such that $\nu(M)\geq \varepsilon$,
one has $\int_X\log \jac_{E_2(g,x)}d\nu(x)<0$.
In particular $g$ belongs to $\cG_\varepsilon(Q,d_1,d_2,d_3)$, and  hence
$f$ does as well (recall that
$f$ belongs to the union of the open sets ${\cG_\varepsilon(Q,d_1,d_2,d_3)}$ and
$\Diff^1_\mathrm{vol}(M)\setminus \overline{\cG_\varepsilon(Q,d_1,d_2,d_3)}$). It follows that $X_f(d_1,d_2,d_3)\cap K$ has volume smaller than $\varepsilon$.
With our choice of $K$, this proves $m(X_f(d_1,d_2,d_3))\leq 2\varepsilon$.
Since $\varepsilon>0$ has been arbitrarily chosen we get $m(X_f(d_1,d_2,d_3))=0$,
as desired.
The induction on $d_2+d_3$ in $\{1,\dots,d-1\}$ concludes the proof in this case.
 \medskip

\noindent
\paragraph{\bf Case 2. The case $\Nuh_f$ has positive volume.}
In the case $\Nuh_f$ has positive volume, we modify the previous argument.
By Theorem~\ref{t.AB}, there exists $d_+,d_-$ such that
$\Nuh_f$ and $X_f(d_+,0,d_-)$ coincide up to a set of volume zero
and by Corollary~\ref{c.vanish}
there exists a global domination $TM=E\oplus F$ with $\dim(E)=d_+$.

\begin{claim}
If $d_2+d_3\leq d_-$, then
for any set $X_f(d'_1,d'_2,d'_3)$ with positive volume, one has $d'_2+d'_3\geq d_2+d_3$.
\end{claim}
\begin{proof}
One considers the three possible case:
\begin{itemize}
\item $(d'_1,d'_2,d'_3)=(0,d,0)$: the claim holds trivially,
\item $d'_1,d'_2,d'_3$ are all non zero: our inductive assumption implies the claim,
\item $d'_2= 0$: this implies $X_f(d'_1,d'_2,d'_3)=\Nuh_f$; hence $d'_2+d'_3=d_-\geq d_2+d_3$.
\end{itemize}
\end{proof}

The induction of case $1$ can thus be repeated while the condition $d_2+d_3\leq d_-$ of the claim holds.
This proves that the Oseledets blocks $X(d_1,d_2,d_3)$ with $d_1,d_2,d_3>0$
and $d_2+d_3\leq d_-$ have measure zero.

Replacing $f$ by $f^{-1}$, one gets the same conclusion for the blocks $X(d_1,d_2,d_3)$ with $d_1,d_2,d_3>0$
and $d_1+d_2\leq d_+$, i.e. such that $d_-\leq d_3$.
This completes the proof in this second case.}
\end{proof}

\subsection{Proof of Theorem B}
Theorem~\ref{t.AB} and  Corollaries~\ref{c.vanish} and~\ref{c.others} now imply Theorem B.

%%%%%%%%%%%%%%%%%%%%%%%%%%%%%%%%%%%%%%%%%%%%%%%%%%%%%%%%%%
%%%%%%%%%%%%%%%%%%%%%%%%%%%%%%%%%%%%%%%%%%%%%%%%%%%%%%%%%%

\section{Local perturbations of  center exponents}\label{s.local}

This section is devoted to the proof of the following,
which implies Theorem C.

\begin{nThm*}{{\bf Theorem C'.}}\label {3.3}
Let $f \in \Diff^1_\mathrm{vol}(M)$, and let $K$ be an $f$-invariant compact set
admitting a dominated splitting $T_{K}M=E_1\oplus E_2\oplus E_3$ into three non-trivial subbundles. Then for any
$\alpha>0$ small and for any neighborhood $\cU\subset \Diff^1_\mathrm{vol}(M)$ of the identity,
there exists $\delta>0$ {\color{black} such that for any $\eta>0$, 
there exists $n_0 \geq 1$ satisfying} the following property.

For any $n\geq n_0$,
any {\color{black} compact} neighborhood $Q$ of $K$ and any {\color{black} $\chi>0$},
there exist a smooth diffeomorphism $\varphi \in \cU$, and a measurable
subset $\Lambda \subset Q$ such that:
\begin{itemize}
\item $\varphi$ is supported on $Q$ and is {\color{black}$\chi$}-close to
the identity in the $C^0$ topology,
\item $m(K \setminus \Lambda)<\eta$,
\item the diffeomorphism $g=f \circ \varphi$ satisfies
\begin{equation}\label{e.decay}
\frac {1} {n} \log \Jac_F(g^n,y) \leq \frac {1} {n} \log\Jac_{E_2(f,y)}(f^n,y)-\delta,
\end{equation}
for every $y \in \Lambda$ such that $y,g^n(y) \in K$, and every
subspace $F  \subset T_y M$ such that $F$ is $\alpha$-close to
$E_2(f,y)$ and $Dg^n(y) \cdot F$ is $\alpha$-close to $E_2(f, g^n(y))$.
\end{itemize}
\end{nThm*}
\medskip

\begin{proof}[Proof of Theorem C from Theorem C']
Consider $f,K,\varepsilon$ as in the statement of Theorem C
and small neighborhoods ${\color{black} \cV}\subset\Diff^1_\mathrm{vol}(M)$ of $f$ and $Q\subset M$ of $K$
such that the maximal invariant set $K(g,Q)$ for any {\color{black} $g\in \cV$} still has a dominated splitting
{\color{black} that extends} the splitting $T_{K}M=E_1\oplus E_2\oplus E_3$ on $K$.
We construct $g$ satisfying the conclusion of the Theorem C.

{\color{black} Let $C_0$ be an upper bound for $d \log \|D g(x)\|$, where $x \in M$, $g\in \cV$.
Fix $\alpha>0$ small.}
Reducing $\cV,Q$ if necessary, for any point
$x\in K(g,Q)\cap K$ the spaces $E_2(f,x)$ and $E_2(g,x)$ are $\alpha$-close.
Theorem C' applied to $\alpha,\cV$, gives {\color{black} $\delta$.
One then chooses $\eta>0$ smaller than
$\min(\varepsilon/10,\delta\varepsilon / 100C_0)$ and Theorem C' gives $n_0$.}
We also take $\kappa>0$
smaller than {\color{black} $\min(\varepsilon/10,{\delta\varepsilon} /{100 C_0})$}.

We choose $n\geq n_0$
and define the {\color{black} compact} set
$$\Omega=\{x \in K,\; \frac {1} {n} \log\Jac_{{\color{black} E_2}(f,x)}(f^n,x) \leq \delta/2\}.$$
If $n$ is large enough,  $K\setminus \Omega$ has measure less than $\kappa$.
For {\color{black} $\chi>0$}  sufficiently small,  shrinking if necessary the neighborhood $Q$,
for any $g$ such that
$g\circ f^{-1}$ is {\color{black} $\chi$}-close to the identity in  the $C^0$ topology, we have:
$$m(K \setminus g^{-n}(K)) \leq \kappa,\quad
m(K(g,Q) \setminus K)\leq \kappa.$$

Theorem C' provides us with a diffeomorphism {\color{black} $g\in \cV$} and a set
$\Lambda$ such that for every $x \in K(g,Q) \cap
K \cap \Lambda \cap \Omega \cap g^{-n}(K)$ one has
$$\frac {1} {n} \log \Jac_{E_2(g,x)}(g^n,x) \leq \frac {1} {n} \log
\Jac_{E_2(f,x)}(f^n,x)-\delta\leq {\color{black} -\delta/2}.$$

Moreover the complement
of the set $Z:=K(g,Q) \cap K \cap \Lambda \cap \Omega \cap g^{-n}(K)$
in $K(g,Q)$ has volume smaller than ${\color{black} 3\kappa+\eta}$.

If $\nu \leq \Leb|Q$ is a $g$-invariant measure with $\nu(M) \geq \varepsilon$,
then $\nu(Z) \geq \varepsilon-3\kappa {\color{black} -\eta}\geq \varepsilon/2$.  Thus
\begin{align*}
\int  \log \Jac_{E_2(g,x)}(g,x) d\nu(x)&= \int
\frac {1} {n} \log
\Jac_{E_2(g,x)}(g^n,x) d\nu(x) \\
\nonumber
&\leq C_0 \nu(M\setminus Z)-\frac {\delta} {2}
\nu(Z)<C_0 {\color{black} (3\kappa+\eta)}-\frac {\delta\varepsilon} {4}<0.\end{align*}
The result follows.
\end{proof}

The construction of the perturbation in Theorem C'
follows three natural steps, and will occupy the remainder of this section.
\bigskip

\subsection{Infinitesimal}
Let $\RR^d=E^+\oplus E^0\oplus E^-$ be an orthogonal decomposition, and
set $d_0=\dim(E^0)$. Let $G
\subset \R^d$ be a two-dimensional subspace that intersects both $E^0$ and
$E^-$ in one-dimensional subspaces, endowed with an arbitrary orientation.
For a subspace $F \subset \R^d$, we let $F^\perp$ denote its orthogonal
complement, and we let $P_F:\R^d \to F$ be the projection with kernel
$F^\perp$. For $\theta \in \R$, let $R_\theta:\R^d \to \R^d$ be the orthogonal operator
that is the identity on $G^\perp$ and that restricted to $G$
is a rotation of angle $2 \pi \theta$ (measured according to the
chosen orientation).

\medskip

\noindent
\emph{Elementary perturbation.}
We introduce a diffeomorphism $\psi^\varepsilon$ which will be used at different places
for the perturbation.
Let $\alpha:\R^d \to \R$ be a smooth function with the following properties:
\begin{itemize}
\item $\alpha(x)=0$ for $x$ in the complement of the unit ball $B:=\{x,\|x\| \geq 1\}$,
\item $\alpha(x)=1$ for $\|x\| \leq 1/2$,
\item $\|\alpha\|_{C^0} \leq 1$,
\item $\alpha(R_\theta \cdot x)=\alpha(x)$ for every $\theta \in \R$ and $x
\in \R^d$.
\end{itemize}
Given $\varepsilon>0$, let $\psi^\varepsilon:\R^d \to \R^d$ be defined by
$\psi^\varepsilon(x)=R_{\varepsilon \alpha(x)} \cdot x$.  It is a smooth,
volume-preserving diffeomorphism
of $\R^d$ and is the identity outside the unit ball. See Figure~\ref{f.phiepsilon}.
{\color{black} We have
$\|\psi^\varepsilon-\id \|_{C^1} \leq \kappa \varepsilon$ for some constant
$\kappa>0$.}
\begin{figure}[h]
\includegraphics[scale=0.35]{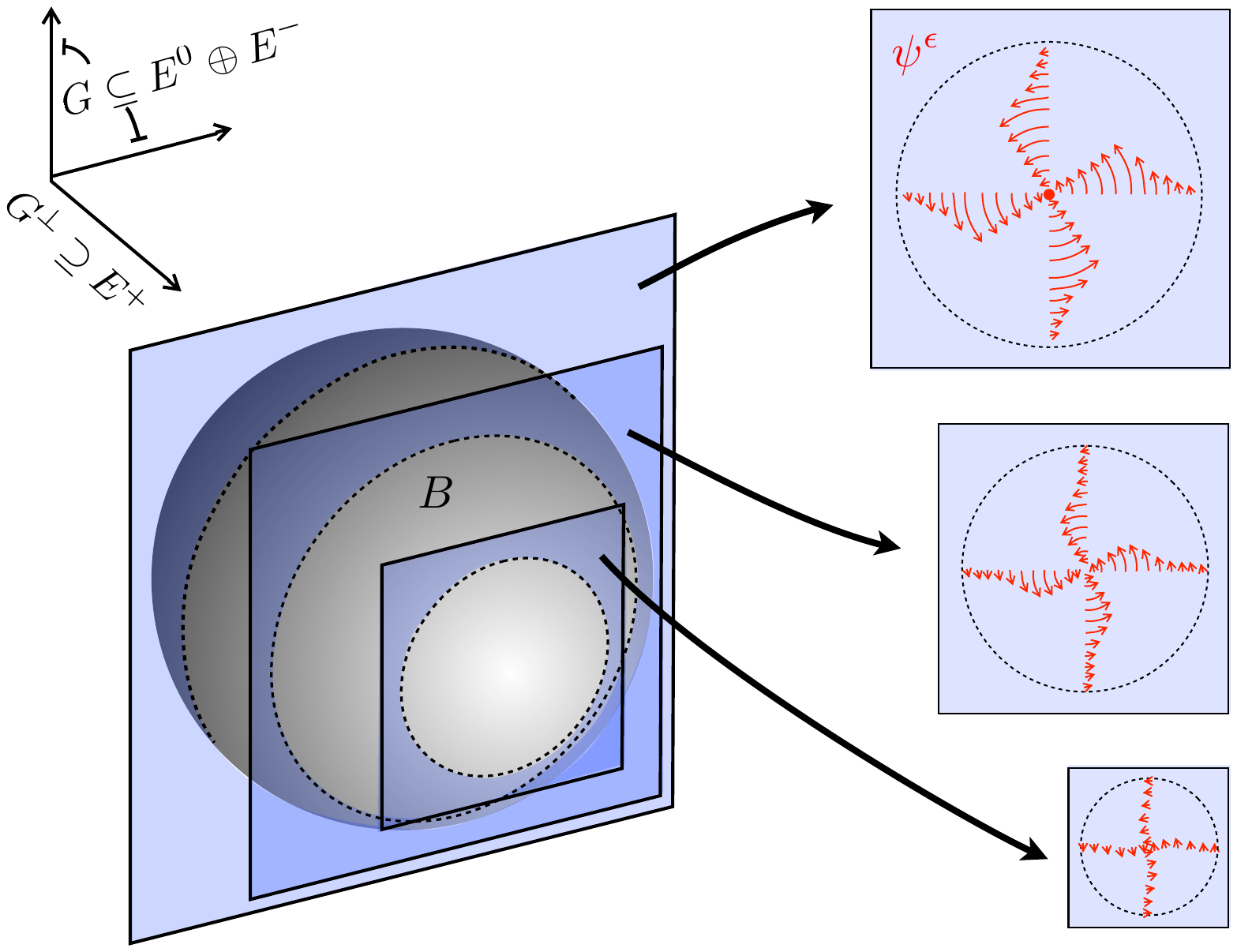}
\caption{\label{f.phiepsilon} The map $\psi^\varepsilon$.}
\end{figure}

Let $\mu_\varepsilon$ be a probability measure in $\SL(d,\R)$ given by the
push-forward under $x \mapsto D \psi^\varepsilon(x)$
of normalized Lebesgue measure $m$ on the unit ball.  Note that for every $A \in
\supp \mu_\varepsilon$, we have $A \cdot (E^0+G)=(E^0+G)$.
We set
\begin{equation}\label{e.c-eps}
 c(\varepsilon)=-\int \log \Jac_{E^0} (P_{E^0} \cdot A) d\mu_\varepsilon(A).
 \end{equation}
{\color{black} Taking $\varepsilon>0$ small enough, the $A\in \supp \mu_\varepsilon$ are close enough to the identity
so that the $\log \Jac_{E^0} (P_{E^0} \cdot A)$ are uniformly bounded. Consequently, $c(\varepsilon)$ is finite.}

We describe the effect of an elementary perturbation averaged on the unit ball.

\begin{lemma}
For every $\varepsilon>0$ sufficiently small, we have $c(\varepsilon)>0$.
\end{lemma}

\begin{proof}
Observe that for any $x_0 \in G^\perp$, $x \mapsto P_{E^0} \cdot
\psi^\varepsilon(x_0+x)$ defines a diffeomorphism of {\color{black} $G$} that is
the identity outside the ball of radius {\color{black} $\max(0,(1-|x_0|^2)^{1/2})$}.  In particular,
Fubini's theorem implies
\begin{equation}
\begin{split}
\int_{\color{black} SL(d,\RR)} \Jac_{E^0} &(P_{E^0} \cdot A) d\mu_\varepsilon(A)=
\int_{\color{black} B} \Jac_{E^0} (P_{E^0} \cdot D\psi^\varepsilon(z)) dm(z)\\
&={\color{black} \int_{G^\perp} \int_G \Jac_{E^0} (P_{E^0} \cdot D\psi^\varepsilon(x_0+x)) \;dx\;dx_0=1.}
\end{split}
\end{equation}
Observe also that for
$|x|<1/2$ we have $\Jac_{E^0}(P_{E^0} \cdot D\psi^\varepsilon(x))=\cos(2 \pi \varepsilon)<1$.
Thus $c(\varepsilon)>0$ follows from Jensen's inequality:
{\color{black}
$$-\int \log \Jac_{E^0} (P_{E^0} \cdot A) d\mu_\varepsilon(A)> -\log \bigg(\int_{\color{black} SL(d,\RR)} \Jac_{E^0} (P_{E^0} \cdot A) d\mu_\varepsilon(A)\bigg) = 0.$$
}
\end{proof}
\medskip

\noindent
\emph{Random composition of elementary perturbations.}
By the Law of Large Numbers,
the effect of an elementary perturbation composed along most random sequences of points of the unit ball is the same as the average effect of a single elementary perturbation.

\begin{proposition} \label{p.3.6}
If $\varepsilon>0$ is small, there exists $\lambda\in (0, 1/4)$
such that for every $\theta>0$ there exist $R_0 \in \N$
{\color{black} and for each $R\geq R_0$
a compact set $W_R \subset \SL(d,\R)^R$ with
$\mu_\varepsilon^{\otimes R}(\SL(d,\R)^R\setminus W_R)<\theta$}
with the following property.
Let $R \geq R_0$ and let $L_j:\R^d \to \R^d$, $0 \leq j \leq R-1$, be
invertible linear operators preserving $E^+$, $E^0$ and $E^-$ such that
$$\|L_j|E^0\| \cdot \|L_j^{-1}|E^+\|\leq \lambda\quad \text{and} \quad
\|L_j|E^-\| \cdot \|L_j^{-1}|E^0\| \leq \lambda.$$
Then
$$
\log \Jac_F\big((L_{R-1} \cdot A_{R-1}) \cdots (L_1\cdot A_1)\cdot (L_0 \cdot A_0)\big) < \sum_{j=0}^{R-1}
\log \Jac_{E^0}(L_j)-\frac {c(\varepsilon)} {2} R,
$$
for every $(A_0,\dots A_{\color{black} R-1})\in W_R$ and
for every $d_0$-dimensional subspace $F$ such that $\|P_{E^-}|F\| \leq 1/2$ and
$\|P_{E^+}|(L_{R-1} \cdot A_{R-1}
\cdots L_0\cdot A_0) \cdot F\| \leq 1/2$.

\end{proposition}
\medskip

The proof will use the following lemma about dominated splittings.
\begin{lemma} \label {3.5}

There exists $C>0$ such that if $\varepsilon>0$ is sufficiently small, then
the following holds.
Let $L:\R^d \to \R^d$ be an invertible linear
operator that preserves each of $E^+$, $E^0$ and $E^-$, and assume that
for some $\lambda\in (0,1/4)$ we have
\begin{equation}\label{e.norm}
\|L|E^0\| \cdot\|L^{-1}|E^+\|\leq \lambda \quad \text{and} \quad
\|L|E^-\| \cdot\|L^{-1}|E^0\|  \leq \lambda.
\end{equation}
Let $A \in \supp \mu_\varepsilon$ and let
$F \subset \R^d$ be a $d_0$-dimensional subspace.  Then
\eqref{e.norm} implies:
\begin{enumerate}
\item[1.] if $\|P_{E^-}|F\| \leq 1/2$ then
$\|P_{E^-}|(L \cdot A) \cdot F\| \leq \lambda$;
\item[2.] if $\|P_{E^+}|(L \cdot A) \cdot F\| \leq 1/2$ then
$\|P_{E^+}|F\| \leq \lambda$; and
\item[3.] if
$\|P_{E^-}|F\|,\|P_{E^+}|(L \cdot A) \cdot F\| \leq \gamma$, for some $\gamma\in(0,1/2)$, then
$$\log \Jac_F(L \cdot A)<\log \Jac_{E^0}(L)+
\log \Jac_{E^0}(P_{E^0} \cdot A)+C (\lambda+\gamma).$$
\end{enumerate}

\end{lemma}

\begin{proof}

If $v \in \R^d$
is a unit vector with $\|P_{E^-} \cdot v\|^2 \leq 1/2$, then
{\color{black} $\|P_{E^-} \cdot v\|\leq \|P_{E^+\oplus E^-} \cdot v\|$. With~\eqref{e.norm} this gives}
$$
\|(P_{E^-} \cdot L) \cdot v\|=
\|(L \cdot P_{E^-}) \cdot v\| \leq
\lambda\, \|(L \cdot P_{E^+ \oplus E^0}) \cdot v\|=
\lambda \,\|(P_{E^+ \oplus E^0} \cdot L) \cdot v\|.
$$
Since $\varepsilon>0$ is small, $\|P_{E^-}|F\| \leq
1/2$ implies $\|P_{E^-}|A \cdot F\|^2 \leq 1/2$.
The first estimate follows.

Symmetrically if $v \in \R^d$
is a unit vector with $\|P_{E^+} \cdot v\|^2 \leq 1/2$, then
$$
\|(P_{E^+} \cdot L^{-1}) \cdot v\|=
\|(L^{-1} \cdot P_{E^+}) \cdot v\| \leq
{\lambda}\,
\|(L^{-1} \cdot P_{E^0 \oplus E^-}) \cdot v\|=
{\lambda}\,
\|(P_{E^0 \oplus E^-} \cdot L) \cdot v\|.
$$
Since $\varepsilon>0$ is small, $\|P_{E^+}|(L\cdot A) \cdot F\| \leq 1/2$
implies $\|P_{E^+}|L\cdot F\|^2 \leq 1/2$.
The second estimate follows.

For any unit vector $v \in \R^d$ such that $\|P_{E^-} \cdot v\|^2
\leq 1/2$ and $\|P_{E^+} \cdot L \cdot v\| \leq \gamma \|L \cdot v\|$,
\begin{equation*}
\begin{split}
\|L \cdot v-(P_{E^0} \cdot L) \cdot v\| & \leq
\|(P_{E^+} \cdot L) \cdot v\|+\|(P_{E^-} \cdot L) \cdot v\|\\
& \leq
\gamma \|L \cdot v\|+
\lambda \|(P_{E^+ \oplus E^0} \cdot L) \cdot v\| \leq
(\gamma + \lambda) \|L \cdot v\|.
\end{split}
\end{equation*}
Thus if $F \subset \R^d$ satisfies $\|P_{E^-}|F\| \leq 1/2$ (and hence
$\|P_{E^-}|A \cdot F\|^2 \leq 1/2$) and
$\|P_{E^+}|(L \cdot A) \cdot F\| \leq \gamma$, we can write
$L|A \cdot F$ as $S_F \cdot L
\cdot (P_{E^0}|A \cdot F)$,
where $S_F:E^0 \to \R^d$ is a linear map with
$\|S_F\| \leq (1-\gamma-\lambda)^{-1}$.  We conclude that
$$
\log \Jac_F(L \cdot A) \leq -d_0 \log (1-\gamma-\lambda)+
\log \Jac_{E^0}(L)+
\log \Jac_F(P_{E^0} \cdot A).
$$
On the other hand, the function $\log \Jac_F(P_{E^0} \cdot A)$ is
uniformly (on $A \in \supp \mu_\varepsilon$)
Lipschitz as a function of  those $F$ satisfying $\|P_{E^+ \oplus E^-}|A \cdot F\|
\leq 1/2$.  Thus
$$
|\log \Jac_F(P_{E^0} \cdot A)-\log \Jac_{E^0}(P_{E^0} \cdot A)| \leq C_0
\|P_{E^+ \oplus E^-}|F\|,
$$
for some $C_0>0$.  Since $\|P_{E^+ \oplus E^-}|F\| \leq
\|P_{E^-}|F\|+\|P_{E^+}|F\| \leq \gamma+\lambda$, the third estimate follows.
\end{proof}
\medskip

\begin{proof}[Proof of Proposition~\ref{p.3.6}]
Define $F_j$, $0 \leq j \leq R$ by $F_0=F$, $F_{j+1}=L_j \cdot A_j \cdot F$.
First notice $\|P_{E^+}|F_R\| \leq 1/2$ and $\|P_{E^-}|F_0\| \leq 1/2$
imply, by iterated application of estimates (1-2) in
the previous lemma, that $\|P_{E^+}|F_j\| \leq \lambda$ for $0 \leq j \leq
R-1$, while $\|P_{E^-}|F_j\| \leq \lambda$ for $1 \leq j \leq R$.  By item
(3) in Lemma~\ref {3.5} we get that $\log \Jac_{F_j}(L_j \cdot A_j)-
(\log \Jac_{E^0}(L_j)+\log \Jac_{E^0}(P_{E^0} \cdot A_j))$ is at most $2 C
\lambda$ if $1 \leq j \leq R-2$, and at most $C \lambda+\frac {C} {2}$ for
$j=0$ or $j=R-1$.  It follows that
\begin{align*}
\log \Jac_F((L_{R-1} \cdot &A_{R-1}) \cdots (L_0 \cdot A_0)) \leq\\
&\sum_{j=0}^{R-1} \log \Jac_{E^0}(L_j) 
\nonumber
+
\sum_{j=0}^{R-1} \Jac_{E^0}(P_{E^0} \cdot A_j)+2C R \lambda+C.
\end{align*}
If $0<\lambda \leq (10 C)^{-1} c(\varepsilon)$ and $R \geq 10 C c(\varepsilon)^{-1}$,
this gives
$$
\log \Jac_F((L_{R-1} \cdot A_{R-1} )\cdots (L_0 \cdot A_0)) \leq \sum_{j=0}^{R-1}
\log \Jac_{E^0}(L_j)+\sum_{j=0}^{R-1} \Jac_{E^0}(P_{E^0} \cdot A_j)+
\frac {3 c(\varepsilon)}{10}R.
$$

Recalling the definition~\eqref{e.c-eps} of $c(\varepsilon)$,
the Law of Large Numbers implies that for every $\theta>0$,
if $R$ is sufficiently large, the probability,
with respect to $\mu_\varepsilon^{\otimes R}$, that
$$
\frac {1} {R} \sum_{j=0}^{R-1} \Jac_{E^0}(P_{E^0} \cdot A_j) \geq
-\frac {4c(\varepsilon)} {5}
$$
is less than $\theta$.  The result follows.
\end{proof}

\subsection{Local}
In the second step, we explain how {\color{black} to} perturb along an orbit.

\begin{proposition} \label {p.3.8}
If $\varepsilon>0$ is small, there exists $\lambda\in (0, 1/4)$
such that for every $\theta>0$ there exists $R_0 \in \N$ with the following property.
Let $R \geq R_0$,
$N \geq R$, and let $f_j:(\R^d,0) \to (\R^d,0)$, $0 \leq j \leq N-1$,
be germs of volume-preserving diffeomorphisms such that
the $L_j=D f_j(0)$ preserve $E^+$,
$E^0$ and $E^-$,
and such that
$$\|L_j|E^0\| \cdot\|L_j^{-1}|E^+\|\leq \lambda \quad \text{and} \quad
\|L_j|E^-\| \cdot\|L_j^{-1}|E^0\| \leq \lambda.$$

Then for every small neighborhood $U$ of $0\in \RR^d$, and $0\leq j\leq N-1$, 
there exist measurable subsets $Z_j$ of $U_j:=f_{j-1}\circ\cdots\circ f_0(U)$,
smooth volume-preserving diffeomorphisms $\varphi_j:\R^d \to \R^d$
and perturbations $\tilde f_j:=f_j\circ \varphi_j$ such that:
\begin{itemize}
\item $m(Z_j) \geq (1-2\theta) m(U_j)$,
\item $\varphi_j$ coincides with $\id$ outside $U_j$ and $D
\varphi_j(x) \in \supp \mu_\varepsilon$ for every $x \in \R^d$,
\item for any $0\leq j\leq N-R$, any $y\in Z_j$ and
any $d_0$-dimensional space $F$ satisfying $\|P_{E^-}|F\| \leq 1/3$
and $\|P_{E^+}|D(\tilde f_{j+R-1}\circ\dots\circ \tilde f_j)(y)\cdot F\| \leq 1/3$,
we have:
$$
\log \Jac_{F}(\tilde f_{j+R-1}\circ\dots\circ \tilde f_j,y) \leq
 \Jac_{E^0}(L_{j+R-1}\circ \dots\circ L_j)-\frac {c(\varepsilon)} {3}R .
$$
\end{itemize}
\end{proposition}

The proof of Proposition~\ref {p.3.8} uses the following lemma, which allows us to construct 
a sequence of perturbations along an orbit that act like random perturbations.
\begin{lemma} \label {3.7}

Consider a sequence $f_j:U_j \to U_{j+1}$, $0\leq j\leq N-1$,
of $C^1$ volume-preserving diffeomorphisms between bounded
open sets of $\R^d$
and $f^j=f_{j-1} \circ \cdots \circ f_0$.  Let $\psi_j$ be
volume-preserving diffeomorphisms of $\RR^d$ supported on the unit ball $B$. 
Let $\mu_j$ be the push-forward of normalized Lebesgue measure $m$ on $B$ under
the map $$B\ni x \mapsto D \psi_j(x) \in \SL(d,\R).$$

Then for any $\chi>0$ there exist {\color{black} orientation- and volume-}preserving diffeomorphisms
$\varphi_j$ of $\RR^d$ such that,
setting $\tilde f_j=f_j \circ \varphi_j$ and $\tilde f^j=\tilde f_{j-1} \circ
\cdots \circ \tilde f_0$, we have:
\begin{enumerate}
\item[1.] for $0 \leq j \leq N-1$, the diffeomorphism $\varphi_j$ is $\chi$-close to the identity
in the $C^0$-distance,
equals $\id$ outside $U_j$, and satisfies
$D \varphi_j(x) \in \supp \mu_j$ for each $x\in \RR^d$;
\item[2.] the push-forward of normalized Lebesgue measure $m$ on $U_0$
under the map
$$U_0 \ni x \mapsto (D \varphi_j(\tilde f^j(x)))_{j=0}^{N-1} \in \SL(d,\R)^N$$
is arbitrarily close to $\mu_0 \otimes \cdots \otimes \mu_{N-1}$.
\end{enumerate}

\end{lemma}

\begin{proof}

The proof is by induction on $N$.  For $N=0$ there is nothing to do.  Assume
it holds for $N-1$, and apply the result for the sequence
$(f_j)_{0\leq j\leq N-2}$, yielding the sequence $(\varphi_j)_{\color{black} 0 \leq j\leq N-2}$.  Define
$\tilde f_j$ and $\tilde f^j$ as before, and let
$\nu_{N-1}$ be the push-forward of normalized Lebesgue measure on
$U_0$ under
the map
$$H_{N-1}:U_0 \ni x \mapsto (D \varphi_j(\tilde f^j(x)))_{j=0}^{N-2} \in
\SL(d,\R)^{N-1},$$ so that $\nu_{N-1}$ is arbitrarily close to $\mu_0 \otimes \cdots \otimes
\mu_{N-2}$.

For $n \in \N$, let $\{D^n_\ell\}_\ell$ be a finite family  of disjoint
closed balls in $U_{N-1}$ {\color{black} chosen using the Vitali lemma} such that:
\begin{itemize}
\item $\diam (D^n_\ell) <n^{-1}$;
\item defining $\hat D^n_\ell\subset U_0$ by
$D^n_\ell=\tilde f^{N-1}(\hat D^n_\ell)$, we have:
$\sum_\ell m(\hat D^n_\ell) \geq (1-n^{-1})\,m(U_0)$;
\item if $x,y \in \hat D^n_\ell$
then $\|H_{N-1}(x)-H_{N-1}(y)\| \leq
n^{-1}$.
\end{itemize}

Let  $\xi_{n,\ell}$ be the conformal affine dilation that sends $B$ into  $D^n_\ell$.
Define $\varphi_{N-1,n}$ to be the
identity outside $\bigcup_\ell D^n_\ell$ and by
$$
\varphi_{N-1,n}(x)=\xi_{n,\ell}
\psi_{N-1}(\xi_{n,\ell}^{-1} x), \quad x \in D_\ell^n.
$$
Let $\nu_{N,n}$
be the push-forward of normalized Lebesgue measure on $U_0$ under
 $$H_{N,n}:U_0 \ni x \mapsto (H_{N-1}{\color{black} (x)},D \varphi_{N-1,n}(\tilde
f^{N-1}(x))) \in \SL(d,\R)^N.$$ 
The properties of the first item are immediate.
{\color{black} For instance $\diam (D^n_\ell) <n^{-1}$ above implies that, for $n$ large enough,
$\varphi_{N-1,n}$ is $C^0$-close to the identity.}

Since
$\nu_{N-1}$ is close to $\mu_0 \otimes \cdots \otimes \mu_{N-2}$,
it is enough to show that $\lim_{n \to \infty}
\nu_{N,n}=\nu_{N-1} \otimes \mu_{N-1}$ to
establish the second item.  Equivalently, we must show
that for a dense subset of compactly supported, continuous functions $\rho:
\SL(d,\R)^N \to \R$, we have
\be
\label {bla}
\lim_{n \to \infty} \int \rho \, d\nu_{N,n}=\int \rho \, d \nu_{N-1}\otimes d \mu_{N-1}.
\ee

Take $\rho$ to be Lipschitz with constant $C_\rho$.  Since $\diam(H_{N-1}(\hat D^n_\ell))
\leq n^{-1}$, the quantities
$$
\frac {1} {m(\hat D^n_\ell)}
\int_{\hat D^n_\ell} \rho(H_{N,n}(x)) dx,
$$
and
$$
\frac {1} {m(\hat D^n_\ell)^2} \int_{\hat D^n_\ell}
\int_{\hat D^n_\ell} \rho(H_{N-1}(x),D \varphi_{N-1,n}(\tilde f^{N-1}(y)))\, dx \, dy
$$
differ by at most $C_\rho n^{-1}$.
By construction, for any {\color{black} $x\in \hat D_\ell^n$} we have
$$
\frac {1} {m(\hat D^n_\ell)}
\int_{\hat D^n_\ell}  \rho(H_{N-1}(x),D \varphi_{N-1}(\tilde f^{N-1}(y)))\, dy=
\int_{SL(d,\RR)} \rho(H_{N-1}(x),z) \, d\mu_{N-1}(z),
$$
so that
\begin{equation}\label {new}
\begin{split}
\left |\int_{\bigcup_\ell \hat D^n_\ell} \rho(H_{N,n}(x))\, dx-
\int \int_{\bigcup_\ell \hat D^n_\ell} \rho(H_{N-1}(x),z)\,
dx \, d\mu_{N-1}(z) \right |\\
\leq
 C_\rho {m\bigg(\bigcup_\ell \hat D^n_\ell\bigg)} n^{-1}.
\end{split}
\end{equation}
Clearly
$$
\left |\int \rho \, d\nu_{N,n}-
\frac {1} {m(U_0)} \int_{\bigcup_\ell \hat D^n_\ell} \rho(H_{N,n}(x)) \, dx \right |
\leq \|\rho\|_\infty n^{-1} \quad \text{and} 
$$
$$
\left |\int \rho\, d\nu_{N-1} \otimes d\mu_{N-1}-
\frac {1} {m(U_0)}\int \int_{\bigcup_l \hat D^n_\ell} \rho(H_{N-1}(x),z)\, dx\,
d\mu_{N-1}(z) \right | \leq
\|\rho\|_\infty n^{-1},
$$
so that (\ref {new}) implies (\ref {bla}).
\end{proof}
\bigskip

\begin{proof}[Proof of Proposition~\ref{p.3.8}]
Use Proposition~\ref {p.3.6} to select $\lambda$, $R_0$ and compact sets $W_R$.
Lemma \ref {3.7} applied with $\psi_j=\psi^\varepsilon$ gives the $\varphi_j$.
In particular, for every $0 \leq j \leq N-R$, there exists $Z_j \subset U_j$ with
$m(Z_j)>(1-2\theta) m(U_j)$ such that the {\color{black} image}  under
$${\color{black} U_j\ni x\mapsto (D \varphi_n(\tilde f^{n-j}(x)))_{n=j}^{j+R-1} \in \SL(d,\R)^R}$$
of the set $Z_j$ is arbitrarily close to $W_R$.
It follows that if $y$ is a point in $Z_j$
and if $F$ is a $d_0$-dimensional space satisfying
$\|P_{E^-}|F\| \leq 1/3$ and $\|P_{E^+}|F'\| \leq 1/3$
for $F'=(L_{j+R-1} \cdot A_{j+R-1}) \cdots (L_j \cdot A_j)\cdot F$,
then

$$
\log \Jac_{F}((L_{j+R-1} \cdot A_{j+R-1}) \cdots (L_j \cdot A_j)) \leq
\log \Jac_{E^0}(L_{j+R-1} \cdots L_j)-\frac {2 c(\varepsilon)} {5} R,
$$
where we denote $A_{j+i}=D \varphi_{j+i}(\tilde f_{j+i-1}\circ\dots\circ \tilde f_j(y_j))$.

Since the $f_{i}$ are diffeomorphisms, if the neighborhood $U$ is small enough,
$$
\log \Jac_{F}(D (\tilde f_{j+{\color{black} R}-1}\circ\dots\circ \tilde f_j)(y_j)) \leq
\log \Jac_{F}((L_{j+R-1} \cdot A_{j+R-1}) \cdots (L_j \cdot A_j))
+\frac {c(\varepsilon)} {20} {\color{black} R}.
$$
The result follows.
\end{proof}

\subsection{Global: proof of Theorem C'}
Using the local perturbation technique along orbits, we define in this third step the global
perturbation by building towers.

\begin{proof}[Proof of Theorem C']
Let $B_\xi\subset \RR^d$ be the ball centered at the origin of radius $\xi>0$ small.
Fix a precompact
family of volume-preserving smooth embeddings $\Psi_x:B_\xi \to M$, $x \in
K$, such that $\Psi_x(0)=x$ and $D\Psi_x(0)$ sends $E^+$, $E^0$, $E^-$
to  $E_1(x)$, $E_2(x)$ and $E_3(x)$, respectively.

Let $\alpha>0$ be small enough so that
(from the dominated splitting $T_KM=E_1\oplus E_2 \oplus E_3$) {\color{black} for all $x\in K$}, if $F$ is
$\alpha$-close to $E_2(x)$ then for each $j\geq 0$
the image $Df^j(x)\cdot F$ is close to a subspace of $E_1(f^j(x))\oplus E_2(f^j(x))$
and $Df^{-j}(x)\cdot F$ is close to a subspace of $E_2(f^{-j}(x))\oplus E_3(f^{-j}(x))$.
In particular for every $j \geq 0$,
$$\|P_{E^+}|
\big(D\Psi_{f^{-j}(x)}(0)^{-1} \cdot Df^{-j}(x)\big) \cdot F\|,\;
\|P_{E^-}|\big(D\Psi_{f^{j}(x)}(0)^{-1} \cdot Df^{j}(x)\big) \cdot F\|\;\leq\; 1/5.$$ 
If $\cU$ is small in the $C^1$-topology, for any $g\in\cU$
and $j\geq 0$ we still have:
\begin{itemize}
\item
if $g(x)$, $g^2(x)$, \dots, $g^j(x)$ are close enough to $f(x)$, $f^2(x)$, \dots, $f^j(x)$, then
$$\|P_{E^-}|
\big(D\Psi_{\color{black} g^j(x)}(0)^{-1} \cdot Dg^j(x)\big) \cdot F\|\; \leq 1/4,$$
\item
if $g^{-1}(x)$, \dots, $g^{-j}(x)$ are close enough to $f^{-1}(x)$, \dots, $f^{-j}(x)$, then
$$\|P_{E^+}|
\big(D\Psi_{\color{black} g^{-j}(x)}(0)^{-1} \cdot Dg^{-j}(x)\big) \cdot F\|\; \leq 1/4.$$
\end{itemize}

We choose $\varepsilon>0$ small (this choice depends on the neighborhood
$\cU$, see below) and apply Proposition~\ref{p.3.8} to get $\lambda$.
The dominated splitting gives  $J_0\in \N$ such that for $x \in K$, the map
$L_x=D\Psi_{f^{J_0}(x)}(0)^{-1} Df^{J_0}(x) D\Psi_x(0)$ satisfies
$$\|L_x|E^0\|\cdot \|L_x^{-1}|E^+\|\leq \lambda\quad 
\text{and}  \|L_x|E^-\|\cdot \|L_x^{-1}|E^0\|\leq \lambda.$$

We then fix $\delta<c(\varepsilon)/(3J_0)$.
Now take $\theta\in (0,\eta/10)$ and apply Proposition~\ref{p.3.8} to get $R_0$.
{\color{black} Next, fix $R$ much larger than $R_0$
(see the choice below) and set $r=R \cdot J_0$.
\medskip

Since $K$ has a dominated splitting, any periodic point $p\in K$
with period $k$ satisfies $Df^k(p)\neq \id$.
The Implicit Function Theorem implies
that the periodic points for $f$ in $K$ have measure $0$.
{\color{black} This implies that there exists a Rokhlin tower}, i.e. a measurable set $Z\subset K$
and a large integer $n_0\geq 1$ such that the iterates
$Z,f(Z),\dots, f^{n_0-1}(Z)$ are pairwise disjoint and
$\cup_{k=0}^{n_0-1} f^k(Z)$ has measure larger than $m(K)-\theta/2$.  {\color{black} Fix such a tower.}
Since $n_0$ is large, one can introduce $n:=N\cdot J_0$ with $N:=[n_0/J_0]$,
and by regularity of the measure, one can replace $Z$
by a compact subset $Y$, so that
$$m(K\setminus \bigcup_{k=0}^{n-1}f^k(Y))<\theta.$$

For each $x\in Y$,
 considers the sequence of diffeomorphisms
$$f_{j,x}:=\Psi_{f^{(j+1) J_0}(x)}^{-1} \circ f^{J_0} \circ
\Psi_{f^{j J_0}(x)},\quad\quad 0 \leq j \leq N-1,$$ and
a neighborhood $D_x$ (which is the image $\Psi_x(U_x)$
of some small neighborhood $U_x$ of $0$).
By compactness, one can find finitely many such points $x_{s} \in Y$, $s\in S$, and
reduce the associated neighborhoods $D_{s}:=D_{x_{s}}$, so that
the $f^k(D_{s})$, $s\in S$, $0\leq k<n$ are pairwise disjoint, and
$$m\bigg(K\setminus \bigcup_{s\in S} \bigcup_{0\leq k<n} f^k(D_{s})\bigg)<{\color{black} 2\theta}.$$
The domains $D_{s}$ may be chosen with small diameter so that
for each point $z\in K$ in an iterate $f^{jJ_0}(D_{s})$, $0\leq j\leq N-1$,
and for any {\color{black} $d_0$-dimensional affine subspace $F\subset \RR^d$,}
\begin{equation}\label{e.continuity}
\color{black} \|P_{E^-}| F\|\leq 1/4
\; \Rightarrow \;
\|P_{E^-}|
D\Psi_{f^{jJ_0}(x_{s})}(0)^{-1} \cdot  D\Psi_{z}(0) \cdot   F\|\leq 1/3,
\end{equation}
and
\begin{equation*}
\color{black} 
\|P_{E^+}| F\|\leq 1/4
\; \Rightarrow \;
\|P_{E^+}|
D\Psi_{f^{(j+R)J_0}(x_{s})}(0)^{-1} \cdot  D\Psi_{f^{RJ_0}(z)}(0) \cdot  F\|\leq 1/3.
\end{equation*}

{\color{black} Proposition~\ref{p.3.8} applied to $x_s$ and to $R,N$ gives
a sequence of diffeomorphisms
$\varphi_{j,s}$, and a sequence of sets $Z_{j,s}\subset f^{jJ_0}(D_s)$ such that
$m(Z_{j,s})\geq (1-2\theta)m(D_{s})$.}
Define the diffeomorphism $\varphi$
in each $f^{j J_0}(D_{s})$, $0 \leq j \leq N-1$ by
$$
\varphi=\Psi_{f^{j J_0}(x_{s})} \circ
\varphi_{j,s} \circ
\Psi_{f^{j J_0}(x_{s})}^{-1},
$$
and let $\varphi=\id$ otherwise.
It is clear that if the neighborhoods ${\color{black} D_{s}}$ are chosen small enough, then $\varphi$ is
arbitrarily close to the identity in the {\color{black}  $C^0$ topology}. 
Also, if $\varepsilon$ is small enough then $\varphi$ is close to the identity
in the {\color{black}  $C^1$ topology}. We set $g=f\circ \varphi$.

Define the set $\Lambda$ to be the set of all points $y$ belonging
to some $f^k(D_{s})$, with $0 \leq k \leq (N-1)J_0-r$, such that
$f^{jJ_0-k}(y) \in Z_{j,s}$, where $j=[k/J_0]+1$.
Hence
$$k\leq jJ_0\leq (j+R)J_0 \leq k+r.$$
Clearly, if $n$ is large and since $10\theta<\eta$,
we have $m(K \setminus \Lambda)<\eta$.
\medskip

Now consider $y\in \Lambda\cap K\cap g^{-r}(K)$
and a $d_0$-dimensional subspace $F\subset T_yM$ that is $\alpha$-close to $E_2(f,y)$
and whose image $Dg^r\cdot F$ is $\alpha$-close to $E_2(f,g^r(y))$.
We also introduce $j,k,x_{s}$ as defined above such that $f^{jJ_0-k}(y)$ belongs to
$Z_{j,s}$.
Since $k-jJ_0$ and $(j+R)J_0-(k+r)$ are bounded (by $2J_0$) and $g$
can be chosen arbitrarily close to $f$ in the $C^1$-topology,
by the choice of $\alpha$ we have
$$\|P_{E^-}|
D\Psi_{f^{jJ_0-k}(y)}(0)^{-1} \cdot D^{jJ_0-k}g(y) \cdot F\|\leq 1/4,$$
$$\|P_{E^+}|D\Psi_{f^{(j+R)J_0-k}(g^r(y))}(0)^{-1} \cdot Dg^{(j+R)J_0-k}(y) \cdot F\|\leq 1/4.$$
By~\eqref{e.continuity}, this gives:
$$\|P_{E^-}|
D\Psi_{f^{jJ_0}(x_{s})}(0)^{-1} \cdot D^{jJ_0-k}g(y) \cdot F\|\leq 1/3,$$
$$\|P_{E^+}|D\Psi_{f^{(j+R)J_0}(x_{s})}(0)^{-1} \cdot Dg^{(j+R)J_0-k}(y) \cdot F\|\leq 1/3.$$

Let $F'=D^{jJ_0-k}g(y)\cdot F$.
Since $f^{jJ_0-k}(y)$ belongs to
$Z_{j,s}$, by applying Proposition~\ref{p.3.8} we obtain:
$$
\log \Jac_{F'}(g^{RJ_0},g^{jJ_0-k}(y)) \leq
\log \Jac_{E_2(f, f^{jJ_0}(x_{s}))}(f^{RJ_0},f^{jJ_0}(x_{s}))-\frac {c(\varepsilon)} {3} R+4C_0,$$
where $C_0$ bounds $|\log\Jac_{H}(D\Psi_x)|$ for any $x\in K$ and any $d_0$-dimensional
space $H$.

If $g$ is sufficiently $C^0$-close to $f$,
and if the sets $D_{s}$ have small diameter,
then the orbits $(f^{-k}(y),\dots, f^{2n-k}(y))$
and $(x_{s},\dots, f^{2n}(x_{s}))$ are arbitrarily close.
It follows that there exists a constant $C_1>0$, which depends on $J_0$ but not
on $R$, such that:
$$
\log \Jac_{F}(g^{r},y) \leq
\log \Jac_{E_2(f, y)}(f^r,y)-\frac {c(\varepsilon)} {3J_0} r+4C_0+C_1.$$
If $r$ (and $R$) has been chosen large enough, one gets~\eqref{e.decay}
by our choice of $\delta$. This ends the proof of Theorem C'.}
\end{proof}

%%%%%%%%%%%%%%%%%%%%%%%%%%%%%%%%%%%%%%%%%%%%%%%%%%%%%%%%%%
%%%%%%%%%%%%%%%%%%%%%%%%%%%%%%%%%%%%%%%%%%%%%%%%%%%%%%%%%%

\section{Proof of the corollaries}

\subsection{Robust positive metric entropy}\label{ss.robustness}
We prove here Corollary 1.

For $m$-almost every point $x$,
we denote the Lyapunov exponents by 
$$\lambda_1(x)\geq\dots\geq \lambda_{\dim M}(x).$$
If $f$ has a (non-trivial) dominated splitting $TM=E\oplus F$, {\color{black} then by the Pesin-type inequality for $C^1$ diffeomorphisms with a dominated splitting proved in ~\cite{sun-tian}, we have:}
{$$h_m(f)\geq \int (\lambda_1(x)+\dots+\lambda_{\dim E}(x))\,d\Leb(x).$$
The dominated splitting also implies that there exists $a>0$
such that $\lambda_{\dim E}(x)>\lambda_{\dim E+1}(x) {\color{black} + a}$ for almost every point $x$.
In particular,
\[a+ \frac 1 {\dim F} \int (\lambda_{\dim E+1}+\dots+\lambda_{\dim M})\,d\Leb< \frac 1 {\dim E}\int (\lambda_1+\dots+\lambda_{\dim E})\,d\Leb.\]
Since $f$ is conservative,
$$\int (\lambda_1+\dots+\lambda_{\dim E})\,d\Leb(x)+ \int (\lambda_{\dim E+1}+\dots+\lambda_{\dim M})\,d\Leb =0.$$
All these estimate together imply that the metric entropy is positive:
$$h_m(f)\geq \int (\lambda_1+\dots+\lambda_{\dim E})\, d\Leb > \frac{a \; \dim E \; \dim F}{\dim M}>0.$$
\medskip

To prove the converse, assume that $f$ has no dominated splitting on $M$.
Then, the Theorem B implies that the generic diffeomorphism $g$
in the open set $\cU$ provided by the lemma below has zero metric entropy.
In particular $f$ is the limit of diffeomorphisms with zero metric entropy.

\begin{lemma}
If $f$ has no dominated splitting on $M$, then
there exists an open set $\cU\subset \Diff^1_\mathrm{vol}(M)$
of diffeomorphisms with no dominated splitting such that $f$ belongs to
the closure of $\cU$.
\end{lemma}
\begin{proof}
Fix $\varepsilon>0$.
There exists~\cite{BC} an arbitrarily $C^1$-small perturbation $f_1$
with a sequence of periodic orbits $O_n$ converging to $M$ in the Hausdorff topology.
Since $f_1$ is arbitrarily close to $f$, the dominated splittings that may exist on $O_n$,
for $n$ large, are weak: by~\cite{BochiBonatti} and the Franks lemma, for each $1\leq i< \dim M$, one can, after
 a $\varepsilon/2$-perturbation $f_2$ (with respect to the $C^1$-distance),
ensure that $O_n$ has simple eigenvalues and that the $i^\text{th}$ and the
$(i+1)^\text{th}$ eigenvalues are complex and conjugated. In particular,
any diffeomorphism $g$ that is $C^1$-close to $f_2$ has no dominated splitting
$E\oplus F$, with $\dim(E)=i$.
This last perturbation is supported on a small neighborhood of $O_n$.
Considering different periodic orbits, one can perform independently
$\dim M-1$ such perturbations and obtain a diffeomorphism which  robustly has
no dominated splitting, as required.
\end{proof}

\subsection{Weak mixing}\label{ss.weak-mixing}
We now prove Corollary 2.

Consider a diffeomorphism $f\in \Diff^r_\mathrm{vol}(M)$ with $r>1$.
For $\Leb$-almost every point $x$ we have introduced in Section~\ref{ss.oseledets} the splitting
$T_xM=E^+(x)\oplus E^0(x)\oplus E^-(x)$ induced by the Oseledets decomposition.
The Pesin stable manifold theorem asserts that {\color{black} if $x\in M$ is a regular point and $\varepsilon>0$ is small, then}
$$W^{-}(x):=\{z: \; \limsup_{n\to +\infty} \frac 1 n\log d(f^n(x),f^n(z))\leq -\varepsilon\}$$
is an injectively immersed submanifold tangent to $E^-(x)$.
Symmetrically, one obtains an injectively immersed submanifold $W^+(x)$ tangent to $E^+(x)$.

If $O$ is a hyperbolic periodic orbit, we define the \emph{Pesin homoclinic class}:

$$\phcs(O) = \{x\hbox{ Oseledets regular }: W^-(x)\transverse W^u(O) \neq \emptyset\},$$
$$\phcu(O) = \{x\hbox{ Oseledets regular }: W^+(x)\transverse W^s(O) \neq \emptyset\},$$
where $W_1\transverse W_2$ denotes the set of transverse intersections between manifolds $W_1,W_2$,
i.e. the set of points $x$ such that $T_xW_1+T_xW_2=T_xM$.
The {Pesin homoclinic class} is $\phc(O) := \phcs(O)\cap \phcu(O)$.
We stress the fact that $\phcs(O)$ can contain points $x$ whose stable dimension $\dim(E^-(x))$
is strictly larger than the stable dimension of $O$. However the set $\phc(O)$
only contains non-uniformly hyperbolic points whose stable/unstable dimensions are the same as $O$.

An improvement of Hopf argument gives:
\begin{theorem}[Rodriguez-Hertz\;-\;Rodriguez-Hertz\;-\;Tahzibi\;-\;Ures~\cite{RRTU}]\label{t.criterion}
Let $f\in \Diff^{r}_m(M)$ with $r>1$
and let $O$ be a hyperbolic periodic orbit such that $m(\phcs(O))$ and $m(\phcu(O))$
are positive.
Then $\phcs(O),\phcu(O),\phc(O)$ coincide $m$-almost everywhere
and $m|\phc(O)$ is  ergodic.  
\end{theorem}

\begin{figure}[h]\label{f=pesin}
\includegraphics[scale=.35]{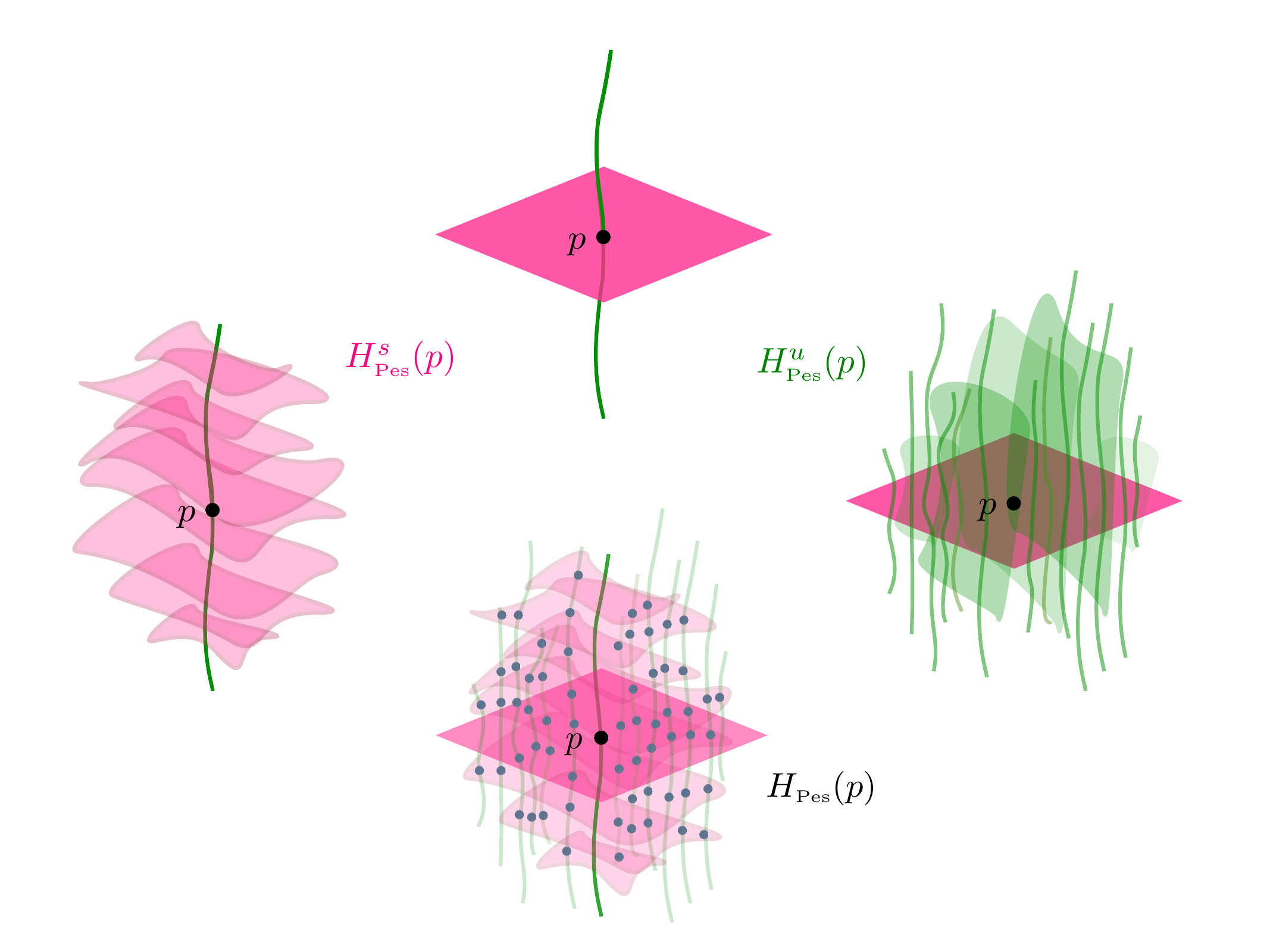}
\caption{The Pesin homoclinic class}
\end{figure}

\medskip

Recall that  $f\in \Diff^1_\mathrm{vol}(M)$   is weakly mixing if and only if
$f\times f$  is ergodic with respect to $m \times m$.

Given  a continuous function $\phi:M\times M \to \R$ and $\epsilon>0$, we denote by $\cU(\phi,\epsilon)$ the set of all $f \in \Diff^1_\mathrm{vol}(M)$ such that, for some $n \geq 1$, the set of all $(x,y) \in M \times M$ satisfying
$$
\left |\frac {1} {n} \sum_{k=0}^{n-1} \phi(f^k(x),f^k(y))-\int \phi(x,y) dm(x) dm(y) \right |<\epsilon
$$
has measure strictly larger than $1-\epsilon$.  Note that  $\cU(\phi,\epsilon)$ is open, and that for any dense subset $\Omega \subset C^0(M \times M,\R)$, $f \times f$ is ergodic if and only if $f$ belongs to $\bigcap_{\phi \in \Omega} \bigcap_{\epsilon>0} \cU(\phi,\epsilon)$.

We say that $f$ is {\em $\epsilon$-weak mixing} if it admits an invariant subset $X$ of measure strictly larger than $1-\epsilon$, such that $f|X$ is weak mixing.  Notice that if $f$ is $\epsilon$-weak mixing then
{\color{black} $f \in \cU(\phi,3 \epsilon\|\phi\|_{L^\infty})$} for every $\phi \in C^0(M \times M,\R)$.  Thus to prove the genericity statement of Corollary 2,  it is enough to prove that $\epsilon$-weak mixing is dense among the diffeomorphisms in $\Diff^1_\mathrm{vol}(M)$ with positive metric entropy.

Let $f\in \Diff^1_\mathrm{vol}(M)$ be a $C^1$-generic diffeomorphism given by Theorem B
and let us assume that it has positive metric entropy.
We may also assume that  $f$ has the following additional $C^1$-generic properties:
\begin{enumerate} 
\item $f$ is topologically transitive, by \cite[Th\'eor\`eme 1.3]{BC},
\item for any hyperbolic periodic point $p$, we have $W^u(p)\cap W^s(f(p))\neq\emptyset$
and the intersection is transverse, by \cite[Theorems 3 and 4]{AC}, and
\item there exist hyperbolic periodic points $p_1,\ldots, p_k$ such that for every $\epsilon>0$ and every $g\in\Diff^2_\mathrm{vol}(M)$ sufficiently $C^1$-close to $f$, there exists a $p_i$ with the following property:  the Pesin homoclinic class $\phc(O(p_i(g)))$ of the orbit $O(p_i)$ of $p_i$ has $m$-measure $>1-\epsilon$ {\color{black} and the restriction of the volume is ergodic, non-uniformly hyperbolic
and its Oseledets splitting is dominated, by \cite[Lemma 5.1]{AB}.}
\end{enumerate}

Let $p_1,\ldots,p_k$ be given by item (3) and let $\epsilon>0$.
By~\cite{A}, there exists
$g\in \Diff^2_\mathrm{vol}(M)$ arbitrarily $C^1$-close to $f$.
Then, by item (2) for each $i=1,\ldots,k$,
there still exists a transverse intersection point
between $W^u(p_i(g))$ and $W^s(g(p_i(g)))$ associated to the hyperbolic continuation $p_i(g)$.  By item (3)
there exists  $i\in 1,\ldots, k$ such that the Pesin homoclinic class  $\phc(O(p_i(g)))$ has $\Leb$-measure $>1-\epsilon$
{\color{black} and the restriction of the volume is ergodic, non-uniformly hyperbolic and its Oseledets splitting is dominated.}

It follows from \cite{Pe} that $\phc(O(p_i(g)))$ decomposes as a disjoint union
of measurable sets $A\cup g(A) \cup\dots \cup g^{\ell-1}(A)$ and that the restriction $m|A$ is Bernoulli
for $g^\ell$. On the other hand, since $W^u(p_i(g))\cap W^s(g(p_i(g)))\neq\emptyset$,
the Pesin homoclinic class of the orbits of $p_i(g)$ for $g$ and $g^\ell$ coincide,
implying by Theorem~\ref{t.criterion} that $m|\phc(O(p_i(g)))$ is ergodic for $g^\ell$.
This shows that $\ell=1$, and that $g$ is Bernoulli, and in particular weakly mixing, on $\phc(O(p_i(g)))$.  Thus $g$ is $\epsilon$-weakly mixing, and so $\epsilon$-weak mixing is dense in $\Diff^1_\mathrm{vol}(M)$.  This completes the proof of Corollary~2.
\qed
 
\bigskip
\noindent
{\color{black}
{\bf Acknowledgements.}\\
We thank Jean-Christophe Yoccoz and the Coll\`ege de France for their kind hospitality.
We are also grateful to the referee for a precise reading that helped us improve the text.}

\end{document}